\documentclass[10pt,a4paper,final]{article}
\usepackage[utf8x]{inputenc}
\usepackage{ucs}
\usepackage{amsmath}
\usepackage{amsfonts}
\usepackage{amsthm}
\usepackage{amssymb}
\usepackage{graphicx}
\usepackage{hyperref}
\usepackage{mathtools}
\usepackage[margin=1in]{geometry}

\usepackage{xcolor}
\usepackage{bbm}
\usepackage{enumitem}
\usepackage{mathabx}
\usepackage{cases}
\usepackage{stmaryrd}

\makeatletter
\DeclareFontFamily{OMX}{MnSymbolE}{}
\DeclareSymbolFont{MnLargeSymbols}{OMX}{MnSymbolE}{m}{n}
\SetSymbolFont{MnLargeSymbols}{bold}{OMX}{MnSymbolE}{b}{n}
\DeclareFontShape{OMX}{MnSymbolE}{m}{n}{
	<-6>  MnSymbolE5
	<6-7>  MnSymbolE6
	<7-8>  MnSymbolE7
	<8-9>  MnSymbolE8
	<9-10> MnSymbolE9
	<10-12> MnSymbolE10
	<12->   MnSymbolE12
}{}
\DeclareFontShape{OMX}{MnSymbolE}{b}{n}{
	<-6>  MnSymbolE-Bold5
	<6-7>  MnSymbolE-Bold6
	<7-8>  MnSymbolE-Bold7
	<8-9>  MnSymbolE-Bold8
	<9-10> MnSymbolE-Bold9
	<10-12> MnSymbolE-Bold10
	<12->   MnSymbolE-Bold12
}{}

\let\llangle\@undefined
\let\rrangle\@undefined
\DeclareMathDelimiter{\llangle}{\mathopen}%
{MnLargeSymbols}{'164}{MnLargeSymbols}{'164}
\DeclareMathDelimiter{\rrangle}{\mathclose}%
{MnLargeSymbols}{'171}{MnLargeSymbols}{'171}
\makeatother

\makeatletter
\newcommand{\opnorm}{\@ifstar\@opnorms\@opnorm}
\newcommand{\@opnorms}[1]{%
	\left|\mkern-1.5mu\left|\mkern-1.5mu\left|
	#1
	\right|\mkern-1.5mu\right|\mkern-1.5mu\right|
}
\newcommand{\@opnorm}[2][]{%
	\mathopen{#1|\mkern-1.5mu#1|\mkern-1.5mu#1|}
	#2
	\mathclose{#1|\mkern-1.5mu#1|\mkern-1.5mu#1|}
}
\makeatother

\usepackage{imakeidx}
\usepackage[nottoc]{tocbibind}
\makeindex[intoc]

\setlist[enumerate]{label=$\rm{(\roman*)}$,leftmargin=\parindent}
\numberwithin{equation}{section}

\newtheorem{thm}{Theorem}
\newtheorem{lem}[thm]{Lemma}
\newtheorem{prop}[thm]{Proposition}

\newtheorem{assume}{Assumption}

\newtheorem{algo}{Algorithm}

\theoremstyle{definition}
\newtheorem{defi}{Definition}
\newtheorem{rmk}{Remark}

\newcommand{\sC}{\mathcal{C}}

\newcommand{\sR}{\mathbb{R}}
\newcommand{\bR}{\mathbf{R}}
\newcommand{\R}{\mathbb{R}}%
\newcommand{\dom}{\mathrm{dom}}

\newcommand{\Ball}{\mathrm{Ball}}
\newcommand{\prox}{\mathrm{prox}}

\newcommand{\crit}{\mathrm{crit}}

\newcommand{\dist}{\mathrm{dist}}
\newcommand{\KL}{K\L }
\newcommand{\Loja}{\L ojasiewicz }

\newcommand{\X}{\mathbf{X}}
\newcommand{\Z}{\mathbf{Z}}
\newcommand{\Y}{\mathbf{Y}}
\newcommand{\E}{\mathcal{E}}
\newcommand{\Id}{\mathrm{Id}}
\newcommand{\0}{\mathbf{0}}
\newcommand{\mysum}{\displaystyle \sum\limits}

\newcommand{\CPsiu}{C_{0}}
\newcommand{\CPsix}{C_{1}}

\newcommand{\Cprex}{C_{2}}
\newcommand{\Cprey}{C_{3}}
\newcommand{\Cpreu}{C_{4}}

\newcommand{\Cdecx}{C_{2}}
\newcommand{\Cdecy}{C_{3}}
\newcommand{\Cdecu}{C_{4}}

\newcommand{\Csgrx}{C_{5}}
\newcommand{\Csgry}{C_{6}}
\newcommand{\Csgru}{C_{7}}

\newcommand{\Cmin}{C_{8}}
\newcommand{\Cmax}{C_{9}}
\newcommand{\Crec}{C_{10}}

\newcommand{\Citex}{C_{11}}
\newcommand{\Citez}{C_{12}}

\title{A proximal minimization algorithm for structured nonconvex and nonsmooth problems}
\author{Radu Ioan Bo\c{t}\footnote{Faculty of Mathematics, University of Vienna, Oskar-Morgenstern-Platz 1, 1090 Vienna, Austria, 
e-mail: \url{radu.bot@univie.ac.at}. Research partially supported by FWF (Austrian Science Fund), project I 2419-N32.} \and 
Ern\"{o} Robert Csetnek\footnote{Faculty of Mathematics, University of Vienna, Oskar-Morgenstern-Platz 1, 1090 Vienna, 
Austria, e-mail: \url{ernoe.robert.csetnek@univie.ac.at}. Research supported by FWF (Austrian Science Fund), project P 29809-N32.} 
\and Dang-Khoa Nguyen\footnote{Faculty of Mathematics, University of Vienna, Oskar-Morgenstern-Platz 1, 1090 Vienna, Austria, 
e-mail: \url{dang-khoa.nguyen@univie.ac.at}. Research supported by the Doctoral Programme 
{\it Vienna Graduate School on Computational Optimization (VGSCO)} which is funded by FWF (Austrian Science Fund), project W1260-N35.}}

\begin{document}
	\maketitle
	
\noindent \textbf{Abstract.} We propose a proximal algorithm for minimizing objective functions consisting of three summands:  the composition of a nonsmooth function with a linear operator, another nonsmooth function, each of the nonsmooth summands depending on an independent block variable, and a smooth function which couples the two block variables. The algorithm is a full splitting method, which means that the nonsmooth functions are processed via their proximal operators, the smooth function via gradient steps, and the linear operator via matrix times vector multiplication. We provide sufficient conditions for the boundedness of the generated sequence and prove that any cluster point of the latter is a KKT point of the minimization problem. In the setting of the Kurdyka-\L{}ojasiewicz property we show global convergence, and derive convergence rates for the iterates in terms of the \L{}ojasiewicz exponent. \vspace{1ex}

\noindent \textbf{Key Words.} structured nonconvex and nonsmooth optimization, proximal algorithm, full splitting scheme, Kurdyka-\L{}ojasiewicz property, limiting subdifferential \vspace{1ex}

\noindent \textbf{AMS subject classification.} 65K10, 90C26, 90C30
	
\section{Introduction}

\subsection{Problem formulation and motivation}

In this paper we propose a full splitting algorithm for solving nonconvex and nonsmooth problems of the form
\begin{equation}
\label{intro:pb}
\min\limits_{\left( x , y \right) \in \sR^{m} \times \sR^{q}} \left\lbrace F \left( Ax \right) + G \left( y \right) + H \left( x , y \right) \right\rbrace ,
\end{equation}
where $F \colon \sR^{p} \to \sR \cup \left\lbrace + \infty \right\rbrace$ and $G \colon \sR^{q} \to \sR \cup \left\lbrace + \infty \right\rbrace$ 
are proper and lower semicontinuous functions, $H \colon \sR^{m} \times \sR^{q} \to \sR$ is a Fr\'{e}chet differentiable function with Lipschitz continuous gradient, and $A \colon \sR^{m} \to \sR^{p}$ is a linear operator. It is noticeable that neither for the nonsmooth nor for the smooth functions convexity is assumed.

In case $m=p$ and $A$ is the identity operator, Bolte, Sabach and Teboulle formulated in \cite{Bolte-Sabach-Teboulle:MP}, also in the nonconvex setting, a proximal alternating linearization method (PALM) for solving \eqref{intro:pb}. PALM is a proximally regularized variant of the Gauss-Seidel alternating minimization scheme and it basically consists of two proximal-gradient steps. It had a significant impact in the optimization community, as it can be used to solve a large variety of nonconvex and nonsmooth problems arising in applications such as: matrix factorization, image deblurring and denoising, the feasibility problem, compressed sensing, etc. An inertial version of PALM has been proposed by Pock and Sabach in \cite{Pock-Sabach}.

A naive approach of PALM for solving \eqref{intro:pb} would require the calculation of the proximal operator of the function $F \circ A$, for which, in general, even in the convex case, a closed formula is not available. 
In the last decade, an impressive progress has been made in the field of primal-dual/proximal ADMM algorithms, designed to solve convex optimization problems involving compositions with linear operators in the spirit of the full splitting paradigm. 
One of the pillars of this development is the conjugate duality theory which is available for convex optimization problems. In addition, several fundamental algorithms, like the proximal method, the forward-backward splitting method, the regularized Gauss-Seidel method, the proximal alternating method, the forward-backward-forward method, and some of their inertial variants have been exported from the convex to the nonconvex setting and proved to convergence globally in the setting of the Kurdyka-\L{}ojasiewicz property (see, for instance, \cite{Attouch-Bolte, Attouch-Bolte-Redont-Soubeyran, Attouch-Bolte-Svaiter, Bolte-Sabach-Teboulle:MP, Bot-Csetnek, Bot-Csetnek-Laszlo}). However, a similar undertaking for structured optimization problems, such as those which involve compositions with linear operators and require for primal-dual methods with a full-splitting character, was by now not very successful. The main reason for that is the absence in the nonconvex setting of a correspondent for the convex  conjugate duality theory.

Despite these premises we succeed to provide in this paper a full splitting algorithm for solving the nonconvex and nonsmooth problem \eqref{intro:pb}; more precisely, the nonsmooth functions are processed via their proximal operators, the smooth function via gradient steps, and the linear operator via matrix times vector multiplication. The convergence analysis is based on  a descent inequality, which we prove for a regularization of the augmented Lagrangian $L_{\beta}:\R^m\times\R^q\times\R^p\times\R^p\to \sR \cup \left\lbrace + \infty\right\rbrace$
$$L_{\beta}(x,y,z,u)=F \left( z \right) + G \left( y \right) + H \left( x , y \right) + \left\langle u , Ax - z \right\rangle + \dfrac{\beta}{2} \left\lVert Ax - z \right\rVert ^{2}, \beta >0,$$
associated with problem \eqref{intro:pb}. This is obtained by an appropriate tuning of the parameters involved in the description of the algorithm. In addition, we provide sufficient conditions in terms of the input functions $F, G$ and $H$ for the boundedness of the generated sequence of iterates. We also show that any cluster point of this sequence is a KKT point of the optimization 
problem \eqref{intro:pb}. By assuming that the above-mentioned regularization of the augmented Lagrangian satisfies the Kurdyka-\L{}ojasiewicz property, we prove global convergence. If this function satisfies the \L{}ojasiewicz property, then we can even derive convergence rates for the sequence of iterates formulated in terms of the \L{}ojasiewicz exponent. For similar approaches based on the use of the Kurdyka-\L{}ojasiewicz property in the proof of the global convergence of nonconvex optimization algorithms we refer to the papers of Attouch and Bolte \cite{Attouch-Bolte}, Attouch, Bolte and Svaiter \cite{Attouch-Bolte-Svaiter}, and Bolte, Sabach and Teboulle \cite{Bolte-Sabach-Teboulle:MP}.  

One of the benefits which comes with the new algorithm is that furnishes a full splitting iterative scheme for the nonsmooth and nonconvex optimization problem
\begin{equation}
\label{intro:pb:particular}
\min\limits_{x \in \sR^{m}} \left\lbrace F \left( Ax \right) + H \left( x \right) \right\rbrace,
\end{equation}
which follows as a particular case of \eqref{intro:pb}  for $G(y) = 0$ and $H(x,y) = H(x)$ for any $(x,y) \in \R^m \times \R^q$, where $H : \R^m \rightarrow \R$ is a Fr\'{e}chet differentiable function with Lipschitz continuous gradient.

In the last years, several articles have been devoted to the design and convergence analysis of algorithms for solving structured optimization problems in the nonconvex and nosmooth setting. 
They all focus on algorithms relying on the alternating direction method of multipliers (ADMM), which is well-known not to be a full splitting algorithm. Nonconvex ADMM algorithms for  \eqref{intro:pb:particular} have been proposed in  \cite{Li-Pong}, 
under the assumption that $H$ is twice continuously differentiable with bounded Hessian, and in \cite{Yang-Pong-Chen}, under the assumption that one of the summands is convex and continuous on its effective domain. In \cite{Wang-Yin-Zeng}, a general nonconvex 
optimization problem involving compositions with linear operators and smooth coupling functions is considered and the importance of providing sufficient conditions for the boundedness of the iterates generated by the proposed  nonconvex ADMM algorithm is recognized. This is achieved by assuming that 
the objective function is continuous and coercive over the feasible set, while its nonsmooth part is either restricted prox-regular or piecewise linear. Similar ingredients are used in \cite{Liu-Shen-Gu} in the convergence analysis of a nonconvex linearized ADMM algorithm. In \cite{Hong-Lou-Razaviyayn}, the ADMM technique is used to minimize the sum of finitely many smooth nonconvex functions and a nonsmooth convex function, by reformulating it as a general consensus problem. In \cite{Wang-Cao-Xu}, a multi-block Bregman ADMM algorithm is proposed and analyzed in a setting based on restrictive strong convexity assumptions. On the other hand, in \cite{Jiang-Lin-Ma-Zhang}, two proximal variants of the ADMM algorithm are introduced and the analyis is focused on providing iteration complexity bounds to reach an $\varepsilon$-KKT solutions. 

We would like to mention in this context also the recent publication \cite{Bolte-Sabach-Teboulle:MOOR} for the case when $A$ is replaced by a nonlinear continuously differentiable operator.

\subsection{Notations and preliminaries}

Every space $\sR^{d}$, where $d$ is a positive integer, is assumed to be equipped with the Euclidean inner product $\left\langle \cdot , \cdot \right\rangle$ and associated norm $\left\lVert \cdot \right\rVert = \sqrt{\left\langle \cdot , \cdot \right\rangle}$. The Cartesian product $\sR^{d_{1}} \times \sR^{d_{2}} \times \ldots \times \sR^{d_{k}}$ of the Euclidean spaces $\R^{d_i}, i=1,...,k$,  will be endowed with inner product and associated norm defined for $x := \left( x_{1} , \ldots , x_{k} \right), y := \left( y_{1} , \ldots , y_{k} \right) 
\in \sR^{d_{1}} \times \sR^{d_{2}} \times \ldots \times \sR^{d_{k}}$ by
\begin{equation*}
\left\llangle x , y \right\rrangle = \mysum_{i=1}^{k} \left\langle x_{i} , y_{i} \right\rangle \quad \textrm{ and } \quad \opnorm{x} = \sqrt{\mysum_{i=1}^{k} \left\lVert x_{i} \right\rVert ^{2}},
\end{equation*}
respectively. For every $x := \left( x_{1} , \ldots , x_{k} \right) \in \sR^{d_{1}} \times \sR^{d_{2}} \times \ldots \times \sR^{d_{k}}$ we have
\begin{equation}
\label{intro:norm-inq}
\dfrac{1}{\sqrt{k}} \mysum_{i=1}^{k} \left\lVert x_{i} \right\rVert \leq \opnorm{x} = \sqrt{\mysum_{i=1}^{k} \left\lVert x_{i} \right\rVert ^{2}} \leq \mysum_{i=1}^{k} \left\lVert x_{i} \right\rVert .
\end{equation}

Let $\psi \colon \sR^{d} \to \sR \cup \left\lbrace + \infty \right\rbrace$ be a proper and lower semicontinuous function and  $x$ an element of its effective domain $\dom \psi := \left\lbrace  y \in \sR^{d} \colon \psi \left( y \right) < + \infty \right\rbrace$. The Fr\'{e}chet (viscosity) subdifferential of $\psi$ at $x$ is
\begin{equation*}
\widehat{\partial} \psi \left( x \right) := \left\lbrace d \in \sR^{d} \colon \liminf\limits_{y \to x} \dfrac{\psi \left( y \right) - \psi \left( x \right) - \left\langle d , y - x \right\rangle}{\left\lVert y - x \right\rVert} \geq 0 \right\rbrace
\end{equation*}
and the limiting (Mordukhovich) subdifferential of $\psi$ at $x$ is
\begin{align*} 
\partial \psi \left( x \right) := \{ d \in \sR^{d} \colon  & \mbox{exist sequences} \ x_{n} \to x \ \mbox{and} \  d_{n} \to d \textrm{ as } n \to +\infty \\
& \mbox{such that } \psi \left( x_{n} \right) \to \psi \left( x \right) \textrm{ as } n \to +\infty \textrm{ and } d_{n} \in \widehat{\partial} \psi \left( x_{n} \right) \ \mbox{for any} \ n \geq 0\} .
\end{align*}
For $x \notin \dom  \psi$, we set $\widehat{\partial} \psi \left( x \right) = \partial \psi \left( x \right) := \emptyset$.

The inclusion $\widehat{\partial} \psi \left( x \right) \subseteq \psi \left( x \right)$ holds for each $x \in \sR^{d}$. 
If $\psi$ is convex, then the two subdifferentials coincide with the convex subdifferential of $\psi$, thus
\begin{equation*}
\widehat{\partial} \psi \left( x \right) = \partial \psi \left( x \right) = \left\lbrace d \in \sR^{d} \colon \psi \left( y \right) \geq \psi\left( x \right) + \left\langle d , y - x \right\rangle \ \forall y \in \sR^{d} \right\rbrace \ \mbox{for any} \ x \in \R^d.
\end{equation*}
If $x \in \sR^{d}$ is a local minimum of $\psi$, then $0 \in \partial \psi \left( x \right)$. 
We denote by $\crit \left( \psi \right) := \left\lbrace x \in \sR^{d} : 0 \in \partial \psi \left( x \right) \right\rbrace$ the set of critical points of $\psi$. 
The limiting subdifferential fulfils the following closedness criterion: if $\left\lbrace x_{n} \right\rbrace _{n \geq 0}$ and $\left\lbrace d_{n} \right\rbrace _{n \geq 0}$ are sequence in $\sR^{d}$ such that $d_{n} \in \partial \psi \left( x_{n} \right)$ for any $n \geq 0$ and $\left( x_{n} , d_{n}  \right) \to \left( x , d \right)$ and $\psi \left( x_{n} \right) \to \psi \left( x \right)$ as $n \to +\infty$, then $d \in \partial \psi \left( x \right)$. 
We also have the following subdifferential sum formula (see \cite[Proposition 1.107]{Mordukhovich}, \cite[Exercise 8.8]{Rockafellar-Wets}): if $\Phi \colon \sR^{d} \to \sR$ is a continuously differentiable function, then $\partial \left( \psi + \phi \right) \left( x \right) = \partial \psi \left( x \right) + \nabla \phi \left( x \right)$ for any $x \in \sR^{d}$; and a formula for the subdifferential of the composition of $\psi$ with a linear operator $A \colon \sR^{k} \to \sR^{d}$ (see \cite[Proposition 1.112]{Mordukhovich}, \cite[Exercise 10.7]{Rockafellar-Wets}): 
if $A$ is injective, then $\partial \left( \psi \circ A \right) \left( x \right) = A^{T} \partial \psi \left( A x \right)$ for any $x \in \R^k$. 

The following proposition collects some important properties of a (not necessarily convex) Fr\'echet differentiable function with Lipschitz continuous gradient. For the proof of this result we refer to \cite[Proposition 1]{Bot-Nguyen}.
\begin{prop}
	\label{prop:lips-semiconvex}
	Let $\psi \colon \sR^{d} \to \sR$ be Fr\'echet differentiable such that its gradient is Lipschitz continuous with constant $\ell  > 0$. Then the following statements are true: 
	\begin{enumerate}
		\item 
		\label{prop:semi}
		For every $x,y \in \sR^{d}$ and every $z \in \left[ x,y \right] = \left\lbrace (1-t)x + ty \colon t \in [0,1] \right\rbrace$ it holds
		\begin{equation}
		\label{eq:semi-convex}
		\psi \left( y \right) \leq \psi \left( x \right) + \left\langle \nabla \psi \left( z \right) , y - x \right\rangle + \dfrac{\ell}{2} \left\lVert y - x \right\rVert ^{2};
		\end{equation}
		
		\item 
		\label{prop:lips}
		For any $\gamma \in \sR \backslash \left\lbrace 0 \right\rbrace$ it holds
		\begin{equation}
		\label{eq:bounded-below}
		\inf\limits_{x \in \sR^{d}} \left\lbrace \psi \left( x \right) - \left( \dfrac{1}{\gamma} - \dfrac{\ell}{2 \gamma^{2}} \right) \left\lVert \nabla \psi \left( x \right) \right\rVert ^{2} \right\rbrace \geq \inf\limits_{x \in \sR^{d}} \psi \left( x \right).
		\end{equation}
	\end{enumerate}	
\end{prop}
The Descent Lemma, which says that for a Fr\'echet differentiable function $\psi \colon \sR^{d} \to \sR$ having a Lipschitz continuous gradient with constant $\ell > 0$ it holds
\begin{equation*} 
\psi \left( y \right) \leq \psi \left( x \right) + \left\langle \nabla \psi \left( x \right) , y - x \right\rangle + \dfrac{\ell}{2} \left\lVert y - x \right\rVert ^{2} \quad \forall x,y \in \sR^{d},
\end{equation*}
follows from \eqref{eq:semi-convex} for $z:=x$.

In addition, by taking in \eqref{eq:semi-convex} $z:=y$ we obtain
\begin{equation*}
\psi \left( x \right) \geq \psi \left( y \right) + \left\langle \nabla \psi \left( y \right) , x - y \right\rangle - \dfrac{\ell}{2} \left\lVert x - y \right\rVert ^{2} \quad \forall x,y \in \sR^{d}.
\end{equation*}
This is equivalent to the fact that $\psi + \dfrac{\ell}{2} \left\lVert \cdot \right\rVert ^{2}$ is a convex function, which is the same with $\psi$ is $\ell$-semiconvex  (\cite{Bolte-Daniilidis-Ley-Mazet}). In other words, a consequence of Proposition \eqref{prop:lips-semiconvex} is, that a Fr\'echet differentiable function with $\ell$-Lipschitz continuous gradient is $\ell$-semiconvex.

We close ths introductory section by presenting two convergence results for real sequences that will be used in the sequel in the convergence analysis. The following lemma is useful when proving convergence of numerical algorithms relying on Fej\'er monotonicity techniques (see, for instance, \cite[Lemma 2.2]{Bot-Csetnek}, \cite[Lemma 2]{Bot-Csetnek-Laszlo}).
\begin{lem}
	\label{lem:conv-pre}
	Let $\left\lbrace \xi_{n} \right\rbrace _{n \geq 0}$ be a sequence of real numbers and $\left\lbrace \omega_{n} \right\rbrace _{n \geq 0}$ a sequence of real nonnegative numbers. Assume that $\left\lbrace \xi_{n} \right\rbrace _{n \geq 0}$ is bounded from below and that for any $n \geq 0$
	\begin{equation*}
	\xi_{n+1} + \omega_{n} \leq \xi_{n}.
	\end{equation*}
	Then the following statements hold:
	\begin{enumerate}
		\item the sequence $\left\lbrace \omega_{n} \right\rbrace _{n \geq 0}$ is summable, namely $\mysum_{n \geq 0} \omega_{n} < + \infty$;
		\item the sequence $\left\lbrace \xi_{n} \right\rbrace _{n \geq 0}$ is monotonically decreasing and convergent.
	\end{enumerate}
\end{lem}

The following lemma can be found in \cite[Lemma 2.3]{Bot-Csetnek} (see, also \cite[Lemma 3]{Bot-Csetnek-Laszlo}).

\begin{lem}
	\label{lem:conv-ext}
	Let $\left\lbrace a_{n} \right\rbrace _{n \geq 0}$ and $\left\lbrace b_{n} \right\rbrace _{n \geq 1}$ be sequences of real nonnegative numbers such that for any $n \geq 1$
	\begin{equation}
	\label{sum:hypo}
	a_{n+1} \leq \chi_{0} a_{n} + \chi_{1} a_{n-1} + b_{n},
	\end{equation}
	where $\chi_{0} \in \sR$ and $\chi_{1} \geq 0$ fulfill $\chi_{0} + \chi_{1} < 1$, and $\mysum_{n \geq 1} b_{n} < + \infty$.
	Then $\mysum_{n \geq 0} a_{n} < + \infty.$
\end{lem}

\section{The algorithm}

The numerical algorithm we propose for solving \eqref{intro:pb} has the following formulation. 

\begin{algo}
	\label{algo:PALM}
	Let $\mu, \beta , \tau > 0$ and $0 < \sigma \leq 1$. For a given starting point $\left( x_{0} , y_{0} , z_{0} , u_{0} \right) \in \sR^{m} \times \sR^{q} \times \sR^{p} \times \sR^{p}$ generate the sequence $\left\lbrace \left( x_{n} , y_{n} , z_{n} , u_{n} \right) \right\rbrace _{n \geq 0}$ for  any $n \geq 0$ as follows
	\begin{subequations}
		\begin{align}
		\begin{split}
		\label{algo:PALM:y}
		y_{n+1} 	& \in \arg\min\limits_{y \in \sR^{q}} \left\lbrace G \left( y \right) + \left\langle \nabla_{y} H \left( x_{n} , y_{n} \right) , y \right\rangle + \dfrac{\mu}{2} \left\lVert y - y_{n} \right\rVert ^{2} \right\rbrace
		\end{split}
		\\
		\begin{split}
		\label{algo:PALM:z}
		z_{n+1} 	& \in \arg\min\limits_{z \in \sR^{p}} \left\lbrace F \left( z \right) + \left\langle u_{n} , Ax_{n} - z \right\rangle + \dfrac{\beta}{2} \left\lVert Ax_{n} - z \right\rVert ^{2} \right\rbrace
		\end{split}
		\\		
		\begin{split}
		\label{algo:PALM:x}
		x_{n+1} 	& := x_{n} - \tau^{-1} \left( \nabla_{x} H \left( x_{n} , y_{n+1} \right) + A^{T} u_{n} + \beta A^{T} \left( Ax_{n} - z_{n+1} \right) \right)
		\end{split}
		\\		
		\begin{split}
		\label{algo:PALM:u}
		u_{n+1} 	& := u_{n} + \sigma \beta \left( Ax_{n+1} - z_{n+1} \right).
		\end{split}		
		\end{align}
	\end{subequations}
\end{algo}
The proximal point operator with parameter $\gamma > 0$ (see \cite{Moreau}) of a proper and lower semicontinuous function $\psi \colon \sR^{d} \to \sR \cup \left\lbrace + \infty \right\rbrace$ is the set-valued operator defined as 
\begin{equation*}
\prox_{\gamma \psi} : \sR^{d} \rightarrow 2^{\sR^{d}}, \quad  \prox_{\gamma \psi} \left( x \right) = \arg\min\limits_{y \in \sR^{d}} \left\lbrace \psi \left( y \right) + \dfrac{1}{2 \gamma} \left\lVert x - y \right\rVert ^{2} \right\rbrace.
\end{equation*}
Exact formulas for the proximal operator are available not only for large classes of convex functions (\cite{bauschke-book,Beck,Combettes-Wajs}), but also for various nonconvex functions  (\cite{Attouch-Bolte-Redont-Soubeyran, Hare-Sagastizabal, Lewis-Malick}).
In view of the above definition, the iterative scheme \eqref{algo:PALM:y} - \eqref{algo:PALM:u} reads for
every $n \geq 0$
\begin{align*}
\begin{split}
y_{n+1}		& \in \prox_{\mu^{-1} G} \left( y_{n} - \mu^{-1} \nabla_{y} H \left( x_{n} , y_{n} \right) \right)  \\
z_{n+1} 	& \in \prox_{\beta^{-1} F} \left( Ax_{n} + \beta^{-1} u_{n} \right) \\
x_{n+1} 	& := x_{n} - \tau^{-1} \left( \nabla_{x} H \left( x_{n} , y_{n+1} \right) + A^{T} u_{n} + \beta A^{T} \left( Ax_{n} - z_{n+1} \right) \right)\\
u_{n+1} 	& := u_{n} + \sigma \beta \left( Ax_{n+1} - z_{n+1} \right).
\end{split}
\end{align*}
One can notice the full splitting character of Algorithm \ref{algo:PALM}  and also that the first two steps can be performed in parallel.

\begin{rmk}\label{remark1}
\begin{enumerate}
\item In case $G(y) = 0$ and $H(x,y) = H(x)$ for any $(x,y) \in \R^m \times \R^q$, where $H : \R^m \rightarrow \R$ is a Fr\'{e}chet differentiable function with Lipschitz continuous gradient, Algorithm \ref{algo:PALM} gives rise to an iterative scheme for solving \eqref{intro:pb:particular} (see also \cite{Bot-Nguyen}) that reads for any $n \geq 0$
\begin{align*}
			\begin{split}
			z_{n+1} 	& \in \prox_{\beta^{-1} F} \left( Ax_{n} + \beta^{-1} u_{n} \right)
			\end{split}
			\\		
			\begin{split}
			x_{n+1} 	& := x_{n} - \tau^{-1} \left( \nabla H \left( x_{n} \right) + A^{T} u_{n} + \beta A^{T} \left( Ax_{n} - z_{n+1} \right) \right)
			\end{split}
			\\		
			\begin{split}
			u_{n+1} 	& := u_{n} + \sigma \beta \left( Ax_{n+1} - z_{n+1} \right).
			\end{split}		
			\end{align*}

\item In case $m=p$ and $A=\Id$ is the identity operator on $\R^m$, Algorithm \ref{algo:PALM} gives rise to an iterative scheme for solving 
\begin{equation}\label{A:Id}
\min\limits_{\left( x , y \right) \in \sR^{m} \times \sR^{q}} \left\lbrace F \left(x \right) + G \left( y \right) + H \left( x , y \right) \right\rbrace,
\end{equation}
 which reads for any $n \geq 0$
\begin{align*}
y_{n+1} & \in \prox_{\mu^{-1} G} \left( y_{n} - \mu^{-1} \nabla_{y} H \left( x_{n} , y_{n} \right) \right) \\
z_{n+1} 	& \in \prox_{\beta^{-1} F} \left(x_{n} + \beta^{-1} u_{n} \right) \\
x_{n+1} 	& := x_{n} - \tau^{-1} \left( \nabla_{x} H \left( x_{n} , y_{n+1} \right) + u_{n} + \beta \left( x_{n} - z_{n+1} \right) \right)\\
u_{n+1} 	& := u_{n} + \sigma \beta \left( x_{n+1} - z_{n+1} \right).
\end{align*}
We notice that, similar to PALM (\cite{Bolte-Sabach-Teboulle:MP}), which is also designed to solve optimization problems of the form \eqref{A:Id}, the algorithm evaluates $F$ and $G$ by proximal steps, while $H$ is evaluated by gradient steps for each of the two blocks.

\item In case $m=p$, $A=\Id$, $F(x) = 0$ and $H(x,y) = H(y)$ for any $(x,y) \in \R^m \times \R^q$, where $H : \R^q \rightarrow \R$ is a Fr\'{e}chet differentiable function with Lipschitz continuous gradient, Algorithm \ref{algo:PALM} gives rise to an iterative scheme for solving
\begin{equation}\label{simple}
\min\limits_{y \in \sR^{q}} \left\lbrace G(y) + H \left( y \right) \right\rbrace,
\end{equation}
which reads for any $n \geq 0$
\begin{align*}
y_{n+1} & \in \prox_{\mu^{-1} G} \left( y_{n} - \mu^{-1} \nabla H(y_n) \right),
\end{align*}
and is nothing else than the proximal-gradient method. An inertial version of the proximal-gradient method for solving \eqref{simple} in the fully nonconvex setting has been considered in \cite{Bot-Csetnek-Laszlo}.
\end{enumerate}
\end{rmk}

\subsection{A descent inequality}

We will start with the convergence analysis of Algorithm \eqref{algo:PALM} by proving a descent inequality, which will play a fundamental role in our investigations. We will analyse Algorithm \eqref{algo:PALM} under the following assumptions, which we will be later even weakened.

\begin{assume}
	\label{as:bound}
	\begin{enumerate}
		\item
		\label{as:bound:i} the functions $F, G$ and $H$ are bounded from below;
		
		\item 
		\label{as:bound:iv}
		the linear operator $A$ is surjective;

		\item 
		\label{as:bound:ii}
		\begin{subequations}
			for any fixed $y \in \sR^{q}$ there exists $\ell_1(y) \geq 0$ such that
			\begin{equation}
			\label{as:lip-xx}
			\left\lVert \nabla_{x} H \left( x , y \right) - \nabla_{x} H \left( x' , y \right) \right\rVert \leq \ell_{1} \left( y \right) \left\lVert x - x' \right\rVert \qquad \forall x , x' \in \sR^{m},
			\end{equation}
			and for any fixed $x \in \sR^{m}$ there exist $\ell_2(x), \ell_3(x) \geq 0$ such that
			\begin{align}
			\begin{split}
			\label{as:lip-yy}
			\left\lVert \nabla_{y} H \left( x , y \right) - \nabla_{y} H \left( x , y' \right) \right\rVert & \leq \ell_{2} \left( x \right) \left\lVert y - y' \right\rVert \qquad \forall y , y' \in \sR^{q} ,
			\end{split}
			\\
			\begin{split}
			\label{as:lip-xy}
			\left\lVert \nabla_{x} H \left( x , y \right) - \nabla_{x} H \left( x , y' \right) \right\rVert & \leq \ell_{3} \left( x \right) \left\lVert y - y' \right\rVert \qquad \forall y , y' \in \sR^{q};
			\end{split}
			\end{align}
		\end{subequations}
		
		\item 
		\label{as:bound:iii}
		there exist $\ell_{i,+} > 0, i=1,2,3,$ such that
		\begin{equation}
		\label{as:sup}
		\sup\limits_{n \geq 0} \ell_{1} \left( y_{n} \right) \leq \ell_{1,+} , \qquad \sup\limits_{n \geq 0} \ell_{2} \left( x_{n} \right) \leq \ell_{2,+} , \qquad \sup\limits_{n \geq 0} \ell_{3} \left( x_{n} \right) \leq \ell_{3,+}.
		\end{equation}

	\end{enumerate}
\end{assume}

\begin{rmk}
	Some comments on Assumption \ref{as:bound} are in order.
	\begin{enumerate}
		\item
Assumption \ref{as:bound:i} ensures that the sequence generated by Algorithm \ref{algo:PALM} is well-defined. It has also as consequence that
		\begin{equation}
		\label{defi:Psi-inf}
		\underline{\Psi} := \inf\limits_{\left( x , y , z \right) \times \sR^{m} \times \sR^{q} \times \sR^{p}} \left\lbrace F \left( z \right) + G \left( y \right) + H \left( x , y \right) \right\rbrace > - \infty .
		\end{equation}
		
		\item
		Comparing the assumptions in (iii) and (iv) to the ones in  \cite{Bolte-Sabach-Teboulle:MP}, one can notice the presence of the additional condition \eqref{as:lip-xy}, which is essential in particular when proving the boundedness of the sequence of generated iterates. Notice that in iterative schemes of gradient type, proximal-gradient type or forward-backward-forward  type (see \cite{Bolte-Sabach-Teboulle:MP,Bot-Csetnek,Bot-Csetnek-Laszlo}) the boundedness of the iterates follow by combining a descent inequality expressed in terms of the objective function with coercivity assumptions on the later. In our setting this undertaken is less simple, since the descent inequality which we obtain below is in terms of the augmented Lagrangian associated with problem \eqref{intro:pb}.

\item 
		The linear operator $A$ is surjective if and only if its associated matrix has full row rank, which is the same with the fact that the matrix associated to $A A^{T}$ is positively definite. Since
		\begin{equation*}
		\lambda_{\min} \left( A A^{T} \right) \left\lVert z \right\rVert ^{2} \leq \langle A A^{T} z, z \rangle = \left\lVert A^{T} z \right\rVert ^{2} \ \forall z \in \sR^{p},
		\end{equation*}
		this is further equivalent to $\lambda_{\min} \left( A A^{T} \right) > 0$, where $\lambda_{\min} \left(M\right)$ denotes the minimal eigenvalue of a square matrix $M$.
		We also denote by $\kappa (M)$ the condition number of $M$, namely the ratio between the maximal eigenvalue $\lambda_{\max} (M)$ and the minimal eigenvalue of the square matrix $M$,
		\begin{equation*}
		\kappa \left( M \right) := \dfrac{\lambda_{\max} \left( M \right)}{\lambda_{\min} \left( M \right)} = \dfrac{\left\lVert M \right\rVert^{2}}{\lambda_{\min} \left( M\right)} \geq 1.
		\end{equation*}
Here, $\|M\|$ denotes the operator norm of $M$ induced by the Euclidean vector norm.
	\end{enumerate}
\end{rmk}

The convergence analysis will make use of the following regularized augmented Lagrangian function
\begin{equation*}
\Psi \colon \sR^{m} \times \sR^{q} \times \sR^{p} \times \sR^{p} \times \sR^{m} \times \sR^{p} \to \sR \cup \left\lbrace + \infty \right\rbrace,
\end{equation*}
defined as
\begin{align*}
\left( x , y , z , u , x' , u' \right) \mapsto & \ F \left( z \right) + G \left( y \right) + H \left( x , y \right) + \left\langle u , Ax - z \right\rangle + \dfrac{\beta}{2} \left\lVert Ax - z \right\rVert ^{2} \nonumber \\
& \ + \CPsiu \left\lVert A^{T} \left( u - u' \right) + \sigma B \left( x - x' \right) \right\rVert ^{2} + \CPsix \left\lVert x - x' \right\rVert ^{2}, 
\end{align*}
where
\begin{equation*}
B := \tau \Id - \beta A^{T} A , \qquad	\CPsiu := \dfrac{4 \left( 1 - \sigma \right)}{\sigma^{2} \beta \lambda_{\min} \left( AA^{T} \right)} \geq 0 \qquad \textrm{ and } \qquad \CPsix := \dfrac{8 \left( \sigma \tau + \ell_{1,+} \right) ^{2}}{\sigma \beta \lambda_{\min} \left( AA^{T} \right)} >0.
\end{equation*}
Notice that
\begin{align*}
\begin{split}
\left\lVert B \right\rVert \leq\tau,
\end{split}
\end{align*}
whenever $2 \tau \geq \beta \left\lVert A \right\rVert ^{2}$. Indeed, this is a consequence of the relation 
$$\|Bx\|^2=\tau^2\|x\|^2-2\tau\beta\|Ax\|^2+\beta^2\|A^TAx\|^2\leq \tau^2\|x\|^2+\beta(\beta\|A\|^2-2\tau)\|Ax\|^2  \ \forall x\in\R^m.$$
For simplification, we introduce the following notations
\begin{align*}
\bR 		& := \sR^{m} \times \sR^{q} \times \sR^{p} \times \sR^{p} \times \sR^{m} \times \sR^{p} \\
\X 			& := \left( x , y , z , u , x' , u' \right) \\
\X_{n} 		& := \left( x_{n} , y_{n} , z_{n} , u_{n} , x_{n-1} , u_{n-1} \right) \ \forall n \geq 1\\
\Psi_{n} 	& := \Psi \left( \X_{n} \right) \ \forall n \geq 1.
\end{align*}

The next result provides the announced descent inequality.

\begin{lem}
	\label{lem:dec}
Let Assumption \ref{as:bound} be satisfied,  $2 \tau \geq \beta \left\lVert A \right\rVert ^{2}$ and $\left\lbrace \left( x_{n} , y_{n} , z_{n} , u_{n} \right) \right\rbrace _{n \geq 0}$ be a sequence generated by Algorithm \ref{algo:PALM}. Then for any $n \geq 1$ it holds
	\begin{equation}
	\label{dec:inq}
	\Psi_{n+1} + \Cprex \left\lVert x_{n+1} - x_{n} \right\rVert ^{2} + \Cprey \left\lVert y_{n+1} - y_{n} \right\rVert ^{2} + \Cpreu \left\lVert u_{n+1} - u_{n} \right\rVert ^{2} \leq \Psi_{n},
	\end{equation}
	where
		\begin{align*}
		\begin{split}
		\Cprex 	& := \tau - \dfrac{\ell_{1,+} + \beta \left\lVert A \right\rVert ^{2}}{2} - \dfrac{4 \sigma \tau^{2}}{\beta \lambda_{\min} \left( AA^{T} \right)} - \dfrac{8 \left( \sigma \tau + \ell_{1,+} \right) ^{2}}{\sigma \beta \lambda_{\min} \left( AA^{T} \right)} ,		
		\end{split}
		\\
		\begin{split}
		\Cprey	& := \dfrac{\mu - \ell_{2,+}}{2} - \dfrac{8 \ell_{3,+}^{2}}{\sigma \beta \lambda_{\min} \left( AA^{T} \right)} ,
		\end{split}
		\\
		\begin{split}
		\Cpreu 	& := \dfrac{1}{\sigma \beta} .
		\end{split}
		\end{align*}
\end{lem}
\begin{proof}
	Let $n \geq 1$ be fixed. We will show first that
	\begin{align}
	& F \left( z_{n+1} \right) + G \left( y_{n+1} \right) + H \left( x_{n+1} , y_{n+1} \right) + \left\langle u_{n+1} , Ax_{n+1} - z_{n+1} \right\rangle + \dfrac{\beta}{2} \left\lVert Ax_{n+1} - z_{n+1} \right\rVert ^{2} \nonumber \\
	& + \left( \tau - \dfrac{\ell_{1,+} + \beta \left\lVert A \right\rVert ^{2}}{2} \right) \left\lVert x_{n+1} - x_{n} \right\rVert ^{2} + \dfrac{\mu - \ell_{2,+}}{2} \left\lVert y_{n+1} - y_{n} \right\rVert ^{2} +\dfrac{1}{\sigma \beta} \left\lVert u_{n+1} - u_{n} \right\rVert ^{2} \nonumber \\
	\leq \ & F \left( z_{n} \right) + G \left( y_{n} \right) + H \left( x_{n} , y_{n} \right) + \left\langle u_{n} , Ax_{n} - z_{n} \right\rangle + \dfrac{\beta}{2} \left\lVert Ax_{n} - z_{n} \right\rVert ^{2} + \dfrac{2}{\sigma \beta} \left\lVert u_{n+1} - u_{n} \right\rVert ^{2} 	\label{dec:pre-inq}
	\end{align}
and provide afterwards an upper estimate for the term $\left\lVert u_{n+1} - u_{n} \right\rVert ^{2}$ on the right-hand side of \eqref{dec:pre-inq}. 
	
From \eqref{algo:PALM:y} and \eqref{algo:PALM:z} we obtain
	\begin{align*}
	G \left( y_{n+1} \right) + \left\langle \nabla_{y} H \left( x_{n} , y_{n} \right) , y_{n+1} - y_{n} \right\rangle + \dfrac{\mu}{2} \left\lVert y_{n+1} - y_{n} \right\rVert ^{2} \leq G \left( y_{n} \right) 
	\end{align*}
	and
	\begin{equation*}
	F \left( z_{n+1} \right) + \left\langle u_{n} , Ax_{n} - z_{n+1} \right\rangle + \dfrac{\beta}{2} \left\lVert Ax_{n} - z_{n+1} \right\rVert ^{2} \leq F \left( z_{n} \right) + \left\langle u_{n} , Ax_{n} - z_{n} \right\rangle + \dfrac{\beta}{2} \left\lVert Ax_{n} - z_{n} \right\rVert ^{2}, 
	\end{equation*}
	respectively.
	Adding these two inequalities yields
	\begin{align}
	& \ F \left( z_{n+1} \right) + G \left( y_{n+1} \right) + \left\langle u_{n} , Ax_{n} - z_{n+1} \right\rangle + \dfrac{\beta}{2} \left\lVert Ax_{n} - z_{n+1} \right\rVert ^{2} +  \left\langle \nabla_{y} H \left( x_{n} , y_{n} \right) , y_{n+1} - y_{n} \right\rangle \nonumber\\ 
&  \  + \dfrac{\mu}{2} \left\lVert y_{n+1} - y_{n} \right\rVert ^{2} \nonumber \\
 \leq & \ F \left( z_{n} \right) + G \left( y_{n} \right) + \left\langle u_{n} , Ax_{n} - z_{n} \right\rangle + \dfrac{\beta}{2} \left\lVert Ax_{n} - z_{n} \right\rVert ^{2}. & \label{dec:1} 
	\end{align}
	On the other hand, according to the Descent Lemma we have
	\begin{align*}
	\begin{split}
	H \left( x_{n} , y_{n+1} \right) & \leq H \left( x_{n} , y_{n} \right) + \left\langle \nabla_{y} H \left( x_{n} , y_{n} \right) , y_{n+1} - y_{n} \right\rangle + \dfrac{\ell_{2} \left( x_{n} \right)}{2} \left\lVert y_{n+1} - y_{n} \right\rVert ^{2} \\
	& \leq H \left( x_{n} , y_{n} \right) + \left\langle \nabla_{y} H \left( x_{n} , y_{n} \right) , y_{n+1} - y_{n} \right\rangle + \dfrac{\ell_{2,+}}{2} \left\lVert y_{n+1} - y_{n} \right\rVert ^{2} 
	\end{split}
	\end{align*}	
	and, further, by taking into consideration \eqref{algo:PALM:x},
	\begin{align*}
	\begin{split}
	H \left( x_{n+1} , y_{n+1} \right) & \leq H \left( x_{n} , y_{n+1} \right) + \left\langle \nabla_{x} H \left( x_{n} , y_{n+1} \right) , x_{n+1} - x_{n} \right\rangle + \dfrac{\ell_{1} \left( y_{n+1} \right)}{2} \left\lVert x_{n+1} - x_{n} \right\rVert ^{2} \\
	& = H \left( x_{n} , y_{n+1} \right) - \left\langle u_{n} , Ax_{n+1} - Ax_{n} \right\rangle - \beta \left\langle Ax_{n} - z_{n+1} , Ax_{n+1} - Ax_{n} \right\rangle \\
	& \quad - \left( \tau - \dfrac{\ell_{1} \left( y_{n+1} \right)}{2} \right) \left\lVert x_{n+1} - x_{n} \right\rVert ^{2} \\
	& \leq H \left( x_{n} , y_{n+1} \right) - \left\langle u_{n} , Ax_{n+1} - Ax_{n} \right\rangle + \dfrac{\beta}{2} \left\lVert Ax_{n} - z_{n+1} \right\rVert ^{2} - \dfrac{\beta}{2} \left\lVert Ax_{n+1} - z_{n+1} \right\rVert ^{2} \\
	& \quad - \left( \tau - \dfrac{\ell_{1,+} + \beta \left\lVert A \right\rVert ^{2}}{2} \right) \left\lVert x_{n+1} - x_{n} \right\rVert ^{2} .
	\end{split}
	\end{align*}
Combining the two above estimates we get
	\begin{align}
	& \ H \left( x_{n+1} , y_{n+1} \right) + \left\langle u_{n} , Ax_{n+1} - Ax_{n} \right\rangle - \dfrac{\beta}{2} \left\lVert Ax_{n} - z_{n+1} \right\rVert ^{2} + \dfrac{\beta}{2} \left\lVert Ax_{n+1} - z_{n+1} \right\rVert ^{2} - \dfrac{\ell_{2,+}}{2} \left\lVert y_{n+1} - y_{n} \right\rVert ^{2} \nonumber \\
	& \ + \left( \tau - \dfrac{\ell_{1,+} + \beta \left\lVert A \right\rVert ^{2}}{2} \right) \left\lVert x_{n+1} - x_{n} \right\rVert ^{2} \nonumber\\
 \leq & \ H \left( x_{n} , y_{n} \right) + \left\langle \nabla_{y} H \left( x_{n} , y_{n} \right) , y_{n+1} - y_{n} \right\rangle .
	\label{dec:2}
	\end{align}		
We obtain \eqref{dec:pre-inq} after we sum up \eqref{dec:1} and \eqref{dec:2}, use \eqref{algo:PALM:u}, and add $\dfrac{2}{\sigma \beta} \left\lVert u_{n+1} - u_{n} \right\rVert^2$ to both sides of the resulting inequality.
	
Next we will focus on estimating  $\left\lVert u_{n+1} - u_{n} \right\rVert ^{2}$. 
We can rewrite \eqref{algo:PALM:x} as
\begin{align*}
\tau \left( x_{n} - x_{n+1} \right) & = \nabla_{x} H \left( x_{n} , y_{n+1} \right) + A^{T} u_{n} + \beta A^{T} \left( Ax_{n+1} - z_{n+1} \right) + \beta A^{T} A \left( x_{n} - x_{n+1} \right) \nonumber \\
& = \nabla_{x} H \left( x_{n} , y_{n+1} \right) + A^{T} u_{n} + \dfrac{1}{\sigma} A^{T} \left( u_{n+1} - u_{n} \right) + \beta A^{T} A \left( x_{n} - x_{n+1} \right),
\end{align*}
where the last equation is due to \eqref{algo:PALM:u}. After multiplying both sides by $\sigma$ and rearranging the terms, we get
\begin{equation*}
A^{T} u_{n+1} + \sigma B \left( x_{n+1} - x_{n} \right) = \left( 1 - \sigma \right) A^{T} u_{n} - \sigma \nabla_{x} H \left( x_{n} , y_{n+1} \right).
\end{equation*}
Since $n$ has been arbitrarily chosen, we also have
\begin{equation*}
A^{T} u_{n} + \sigma B \left( x_{n} - x_{n-1} \right) = \left( 1 - \sigma \right) A^{T} u_{n-1} - \sigma \nabla_{x} H \left( x_{n-1} , y_{n} \right).
\end{equation*}
Subtracting these relations and making use of the notations
	\begin{align*}
	\begin{split}
	w_{n}	& := A^{T} \left( u_{n} - u_{n-1} \right) + \sigma B \left( x_{n} - x_{n-1} \right) \\
	v_{n}	& := \sigma B \left( x_{n} - x_{n-1} \right) + \nabla_{x} H \left( x_{n-1} , y_{n} \right) - \nabla_{x} H \left( x_{n} , y_{n+1} \right) ,
	\end{split}
	\end{align*}
it yields
	\begin{equation*}
	w_{n+1} = \left( 1 - \sigma \right) w_{n} + \sigma v_{n} .
	\end{equation*}
The convexity of $\left\lVert \cdot \right\rVert ^{2}$ guarantees that (notice that $0< \sigma \leq 1$)
	\begin{equation}
	\label{dec:convex}
	\left\lVert w_{n+1} \right\rVert ^{2} \leq \left( 1 - \sigma \right) \left\lVert w_{n} \right\rVert ^{2} + \sigma \left\lVert v_{n} \right\rVert ^{2} .
	\end{equation}
In addition, from the definitions of $w_{n}$ and $v_{n}$, we obtain
	\begin{equation}
	\label{dec:pre-Wn}
	\left\lVert A^{T} \left( u_{n+1} - u_{n} \right) \right\rVert \leq \left\lVert w_{n+1} \right\rVert + \sigma \left\lVert B \right\rVert \left\lVert x_{n+1} - x_{n} \right\rVert \leq \left\lVert w_{n+1} \right\rVert + \sigma \tau \left\lVert x_{n+1} - x_{n} \right\rVert
	\end{equation}
	and
	\begin{align}
	\left\lVert v_{n} \right\rVert & \leq \sigma \left\lVert B \right\rVert \left\lVert x_{n} - x_{n-1} \right\rVert + \left\lVert \nabla_{x} H \left( x_{n-1} , y_{n} \right) - \nabla_{x} H \left( x_{n} , y_{n+1} \right) \right\rVert \nonumber \\
	& \leq \sigma \tau \left\lVert x_{n} - x_{n-1} \right\rVert + \left\lVert \nabla_{x} H \left( x_{n-1} , y_{n} \right) - \nabla_{x} H \left( x_{n} , y_{n} \right) \right\rVert + \left\lVert \nabla_{x} H \left( x_{n} , y_{n} \right) - \nabla_{x} H \left( x_{n} , y_{n+1} \right) \right\rVert \nonumber \\
	& \leq \left( \sigma \tau + \ell_{1,+} \right) \left\lVert x_{n} - x_{n-1} \right\rVert + \ell_{3,+} \left\lVert y_{n+1} - y_{n} \right\rVert 	\label{dec:pre-Vn}
	\end{align}
respectively.  Using the Cauchy-Schwarz inequality, \eqref{dec:pre-Wn} yields
	\begin{equation*}
	\dfrac{\lambda_{\min} \left( AA^{T} \right)}{2} \left\lVert u_{n+1} - u_{n} \right\rVert ^{2} \leq \dfrac{1}{2} \left\lVert A^{T} \left( u_{n+1} - u_{n} \right) \right\rVert ^{2} \leq \left\lVert w_{n+1} \right\rVert ^{2} + \sigma^{2} \tau^{2} \left\lVert x_{n+1} - x_{n} \right\rVert ^{2}
	\end{equation*}
	and  \eqref{dec:pre-Vn} yields
	\begin{equation*}
	\left\lVert v_{n} \right\rVert ^{2} \leq 2 \left( \sigma \tau + \ell_{1,+} \right) ^{2} \left\lVert x_{n} - x_{n-1} \right\rVert ^{2} + 2 \ell_{3,+}^{2} \left\lVert y_{n+1} - y_{n} \right\rVert ^{2}.
	\end{equation*}
After combining these two inequalities with \eqref{dec:convex}, we get
	\begin{align*}
	\begin{split}
	& \dfrac{\sigma \lambda_{\min} \left( AA^{T} \right)}{2} \left\lVert u_{n+1} - u_{n} \right\rVert ^{2} + \left( 1 - \sigma \right) \left\lVert w_{n+1} \right\rVert ^{2} \\
	\leq & \left( 1 - \sigma \right) \left\lVert w_{n} \right\rVert ^{2} + \sigma^{3} \tau^{2} \left\lVert x_{n+1} - x_{n} \right\rVert ^{2} + 2 \sigma \left( \sigma \tau + \ell_{1,+} \right) ^{2} \left\lVert x_{n} - x_{n-1} \right\rVert ^{2} + 2 \sigma \ell_{3,+}^{2} \left\lVert y_{n+1} - y_{n} \right\rVert ^{2} .
	\end{split}
	\end{align*}

After multiplying the above relation by $\dfrac{4}{\sigma^{2} \beta \lambda_{\min} \left( AA^{T} \right)} > 0$ and adding the resulting inequality to \eqref{dec:pre-inq} it yields
\begin{align*}\allowdisplaybreaks
	& F \left( z_{n+1} \right) + G \left( y_{n+1} \right) + H \left( x_{n+1} , y_{n+1} \right) + \left\langle u_{n+1} , Ax_{n+1} - z_{n+1} \right\rangle + \dfrac{\beta}{2} \left\lVert Ax_{n+1} - z_{n+1} \right\rVert ^{2} \nonumber \\
	& + \frac{4(1-\sigma)}{\sigma^2 \beta \lambda_{\min}(AA^T)}\|A^T(u_{n+1}-u_{n}) + \sigma B(x_{n+1}-x_n)\|^2 + \dfrac{8 \left( \sigma \tau + \ell_{1,+} \right) ^{2}}{\sigma \beta \lambda_{\min} \left( AA^{T} \right)} \|x_{n+1} - x_n\|^2 \nonumber \\
& + \!\left( \tau - \dfrac{\ell_{1,+} + \beta \left\lVert A \right\rVert ^{2}}{2}  -  \sigma^3\tau^2 - \dfrac{8 \left( \sigma \tau + \ell_{1,+} \right) ^{2}}{\sigma \beta \lambda_{\min} \left( AA^{T} \right)}  \right)\! \left\lVert x_{n+1} - x_{n} \right\rVert ^{2} \!\nonumber \\
& + \left(\dfrac{\mu - \ell_{2,+}}{2} - \dfrac{8 \ell_{3,+}^{2}}{\sigma \beta \lambda_{\min} \left( AA^{T} \right)}\right)\!\left\lVert y_{n+1} - y_{n} \right\rVert ^{2}  +\dfrac{1}{\sigma \beta} \left\lVert u_{n+1} - u_{n} \right\rVert ^{2} \nonumber\\
\leq \ & F \left( z_{n} \right) + G \left( y_{n} \right) + H \left( x_{n} , y_{n} \right) + \left\langle u_{n} , Ax_{n} - z_{n} \right\rangle + \dfrac{\beta}{2} \left\lVert Ax_{n} - z_{n} \right\rVert ^{2} \nonumber \\
& + \frac{4(1-\sigma)}{\sigma^2 \beta \lambda_{\min}(AA^T)}\|A^T(u_{n}-u_{n-1}) + \sigma B(x_{n}-x_{n-1})\|^2 + \dfrac{8 \left( \sigma \tau + \ell_{1,+} \right) ^{2}}{\sigma \beta \lambda_{\min} \left( AA^{T} \right)} \|x_{n} - x_{n-1}\|^2,
	\end{align*}
which is nothing else than \eqref{dec:inq}.
\end{proof}

The following result provides one possibility to choose the parameters in Algorithm \ref{algo:PALM}, such that all three constants $C_2, C_3$ and $C_4$ that appear in \eqref{dec:inq} are positive.

\begin{lem}
	\label{lem:pos}
	\begin{subequations}
		Let
		\begin{equation}
		\label{pos:sigma}
		0 < \sigma < \dfrac{1}{24 \kappa \left( AA^{T} \right)}
		\end{equation}
		\begin{equation}
		\label{pos:beta}
		\beta > \dfrac{\nu}{1 - 24 \sigma \kappa \left( AA^{T} \right)} \left( 4 + 3 \sigma +  \sqrt{24 + 24 \sigma + 9 \sigma^{2} - 192 \sigma \kappa \left( AA^{T} \right)} \right) > 0
		\end{equation}
		\begin{equation}
		\label{pos:tau}
		\max \left\lbrace \dfrac{\beta \left\lVert A \right\rVert ^{2}}{2} , \dfrac{\beta \lambda_{\min} \left( AA^{T} \right)}{24 \sigma} \left( 1 - \dfrac{4 \nu}{\beta} - \sqrt{\Delta_{\tau} '} \right) \right\rbrace < \tau < \dfrac{\beta \lambda_{\min} \left( AA^{T} \right)}{24 \sigma} \left( 1 - \dfrac{4 \nu}{\beta} + \sqrt{\Delta_{\tau} '} \right)
		\end{equation}
		\begin{equation}
		\label{pos:mu}
		\mu > \ell_{2,+} + \dfrac{16 \ell_{3,+}^{2}}{\sigma \beta \lambda_{\min} \left( AA^{T} \right)} > 0,
		\end{equation}
		where
		\begin{equation*}
		\nu 	:= \dfrac{4 \ell_{1,+}}{\lambda_{\min} \left( AA^{T} \right)} > 0 \qquad \textrm{ and } \qquad
		\Delta_{\tau} '  := 1 - \dfrac{8 \nu}{\beta} - \dfrac{8 \nu^{2}}{\beta^{2}} - \dfrac{6 \nu \sigma}{\beta} - 24 \sigma \kappa \left( AA^{T} \right) > 0 .
		\end{equation*}
		Then we have
		\begin{equation*}
		\min \left\lbrace \Cprex , \Cprey , \Cpreu \right\rbrace > 0 .
		\end{equation*}		
	\end{subequations}
Furthermore, there exist $\gamma_{1} , \gamma_{2} \in \sR \backslash \left\lbrace 0 \right\rbrace$ such that
\begin{equation}
\label{pos:gamma}
\dfrac{1}{\gamma_{1}} - \dfrac{\ell_{1,+}}{2 \gamma_{1}^{2}} = \dfrac{1}{\beta \lambda_{\min} \left( AA^{T} \right)} \qquad \textrm{ and } \qquad \dfrac{1}{\gamma_{2}} - \dfrac{\ell_{1,+}}{2 \gamma_{2}^{2}} = \dfrac{2}{\beta \lambda_{\min} \left( AA^{T} \right)} .
\end{equation}
\end{lem}
\begin{proof}
We will prove first that $\Cprex > 0$ or, equivalently,
\begin{equation}
\label{pos:Cdecx}
-2 \Cprex = \dfrac{24 \sigma \tau^{2}}{\beta \lambda_{\min} \left( AA^{T} \right)} - 2 \left( 1 - \dfrac{16 \ell_{1,+}}{\beta \lambda_{\min} \left( AA^{T} \right)} \right) \tau + \dfrac{16 \ell_{1,+}^{2}}{\sigma \beta \lambda_{\min} \left( AA^{T} \right)} + \ell_{1,+} + \beta \left\lVert A \right\rVert ^{2} < 0 .
\end{equation}
The reduced discriminant of the quadratic function in $\tau$ in the above relation fulfils 
\begin{align}
\Delta_{\tau} ' & := \left( 1 - \dfrac{16 \ell_{1,+}}{\beta \lambda_{\min} \left( AA^{T} \right)} \right) ^{2} - \dfrac{384 \ell_{1,+}^{2}}{\beta^{2} \lambda_{\min}^{2} \left( AA^{T} \right)} - \dfrac{24 \ell_{1,+} \sigma}{\beta \lambda_{\min} \left( AA^{T} \right)} - 24 \sigma \kappa \left( AA^{T} \right) \nonumber \\
& = \left( 1 - \dfrac{4 \nu}{\beta} \right) ^{2} - \dfrac{24 \nu^{2}}{\beta^{2}} - \dfrac{6 \nu \sigma}{\beta} - 24 \sigma \kappa \left( AA^{T} \right) \nonumber \\
& = 1 - \dfrac{8 \nu}{\beta} - \dfrac{8 \nu^{2}}{\beta^{2}} - \dfrac{6 \nu \sigma}{\beta} - 24 \sigma \kappa \left( AA^{T} \right) > 0, \label{pos:dis}
\end{align}
if $\sigma$ and $\beta$ are being chosen as in \eqref{pos:sigma} and \eqref{pos:beta}, respectively.
Indeed, the inequality \eqref{pos:dis} is equivalent to
\begin{equation*}
\left( 1 - 24 \sigma \kappa \left( AA^{T} \right) \right) \beta^{2} - 2 \left( 4 + 3 \sigma \right) \nu \beta - 8 \nu^{2} > 0.
\end{equation*}
The reduced discriminant of the quadratic function in $\beta$ in the above relation  reads
\begin{equation*}
\Delta_{\beta} := \left[ \left( 4 + 3 \sigma \right) ^{2} + 8 \left( 1 - 24 \sigma \kappa \left( AA^{T} \right) \right) \right] \nu^{2} = \left[ 24 + 24 \sigma + 9 \sigma^{2} - 192 \sigma \kappa \left( AA^{T} \right) \right] \nu^{2} > 0
\end{equation*}
as $24 - 192 \sigma \kappa \left( AA^{T} \right) = 16 + 8 \left( 1 - 24 \sigma \kappa \left( AA^{T} \right) \right) > 0$ for every $\sigma$ that satisfies \eqref{pos:sigma}. Hence, for every $\sigma$ satisfying \eqref{pos:sigma} and every $\beta$ satisfying \eqref{pos:beta} it holds \eqref{pos:dis}.
Therefore, \eqref{pos:Cdecx} is satisfied for every
\begin{equation*}
\dfrac{\beta \lambda_{\min} \left( AA^{T} \right)}{24 \sigma} \left( 1 -\dfrac{4 \nu}{\beta} - \sqrt{\Delta_{\tau} '} \right) < \tau < \dfrac{\beta \lambda_{\min} \left( AA^{T} \right)}{24 \sigma} \left( 1 - \dfrac{4 \nu}{\beta} + \sqrt{\Delta_{\tau} '} \right).
\end{equation*}

It remains to verify the feasibility of $\tau$ in \eqref{pos:tau}, in other words, to prove that
\begin{equation*}
\dfrac{\beta \left\lVert A \right\rVert ^{2}}{2} < \dfrac{\beta \lambda_{\min} \left( AA^{T} \right)}{24 \sigma} \left( 1 - \dfrac{4 \nu}{\beta} + \sqrt{\Delta_{\tau} '} \right) .
\end{equation*}
This is easy to see, as, according to \eqref{pos:dis}, we have
\begin{equation*}
\dfrac{\beta \left\lVert A \right\rVert ^{2}}{2} < \dfrac{\beta \lambda_{\min} \left( AA^{T} \right)}{24 \sigma} \left( 1 - \dfrac{4 \nu}{\beta} \right) \Leftrightarrow 1 - \dfrac{4 \nu}{\beta} - 12 \sigma \kappa \left( AA^{T} \right) > 0.
\end{equation*}
The positivity of $\Cprey$ follows from the choice of $\mu$ in \eqref{pos:mu}, while, obviously,  $\Cpreu > 0$.

Finally, we notice that the reduced discriminants of the two quadratic equations in \eqref{pos:gamma} (in $\gamma_{1}$ and, respectively, $\gamma_{2}$) are
\begin{equation*}
\Delta_{\gamma_{1}} := 1 - \dfrac{2 \ell_{1,+}}{\beta \lambda_{\min} \left( AA^{T} \right)} = 1 - \dfrac{\nu}{2 \beta} \quad \textrm{ and, respectively, } \quad \Delta_{\gamma_{2}} := 1 - \dfrac{\ell_{1,+}}{\beta \lambda_{\min} \left( AA^{T} \right)} = 1 - \dfrac{\nu}{4 \beta}.
\end{equation*}
Since
\begin{equation*}
\beta > \dfrac{\nu}{1 - 24 \sigma \kappa \left( AA^{T} \right)} > \dfrac{\nu}{2},
\end{equation*} 
it follows that $\Delta_{\gamma_{1}}, \Delta_{\gamma_{2}} > 0$ and hence each of the two equations has a nonzero real solution.
\end{proof}

\begin{rmk}
Hong and Luo proved recently in \cite{Hong-Luo} linear convergence for the iterates generated by a Lagrangian-based algorithm in the convex setting, without any strong convexity assumption. To this end a certain error bound condition must hold true and the step size of the dual update, which is also assumed to depend on the error bound constants, must be taken small. It is also mentioned that  the dual step size may be cumbersome to compute unless the objective function is strongly convex. As one can see in \eqref{pos:sigma} and \eqref{pos:beta}, the step size of the dual update in our algorithm can be chosen only in dependence of the condition number of $AA^T$.
\end{rmk}

\begin{thm}
	\label{thm:lim}
Let Assumption \ref{as:bound} be satisfied and the parameters in Algorithm \ref{algo:PALM} be such that $2 \tau \geq \beta \left\lVert A \right\rVert ^{2}$ and the constants defined in Lemma \ref{lem:dec} fulfil $\min\{C_2, C_3, C_4\} >0$. If $\left\lbrace \left( x_{n} , y_{n} , z_{n} , u_{n} \right) \right\rbrace _{n \geq 0}$ is a sequence generated by Algorithm \ref{algo:PALM}, then the following statements are true:
	\begin{enumerate}
		\item 
		\label{lim:i}
		the sequence $\left\lbrace \Psi_{n} \right\rbrace_{n \geq 1}$ is bounded from below and convergent;		
		\item
		\label{lim:ii}
		\begin{equation}
		\label{lim:vanish}
		x_{n+1} - x_{n} \to 0 , \ y_{n+1} - y_{n} \to 0 , \ z_{n+1} - z_{n} \to 0 \textrm{ and } \ u_{n+1} - u_{n} \to 0 \textrm{ as } \ n \to + \infty .
		\end{equation}
	\end{enumerate}
\end{thm}
\begin{proof}
First, we show that $\underline{\Psi}$ defined in \eqref{defi:Psi-inf} is a lower bound of $\left\lbrace \Psi_{n} \right\rbrace_{n \geq 2}$. Suppose the contrary, namely that there exists $n_ {0} \geq 2$ such that $\Psi_{n_{0}} - \underline{\Psi} < 0$. According to Lemma \ref{lem:dec}, $\left\lbrace \Psi_{n} \right\rbrace_{n \geq 1}$ is a nonincreasing sequence and thus for any $N \geq n_{0}$
	\begin{align*}
	\mysum_{n=1}^{N} \left( \Psi_{n} - \underline{\Psi} \right) \leq \mysum_{n=1}^{n_{0}-1} \left( \Psi_{n} - \underline{\Psi} \right) + \left( N - n_{0} + 1 \right) \left( \Psi_{n_{0}} - \underline{\Psi} \right) ,
	\end{align*}
	which implies that
	\begin{equation*}
	\lim\limits_{N \to + \infty} \mysum_{n=1}^{N} \left( \Psi_{n} - \underline{\Psi} \right) = - \infty .
	\end{equation*}
	On the other hand, for any $n \geq 1$ it holds
	\begin{align*}
	\begin{split}
	\Psi_{n} - \underline{\Psi} & \geq F \left( z_{n} \right) + G \left( y_{n} \right) + H \left( x_{n} , y_{n} \right) + \left\langle u_{n} , Ax_{n} - z_{n} \right\rangle - \underline{\Psi} \\
	& \geq \left\langle u_{n} , Ax_{n} - z_{n} \right\rangle = \dfrac{1}{\sigma \beta} \left\langle u_{n} , u_{n} - u_{n-1} \right\rangle = \dfrac{1}{2 \sigma \beta} \left\lVert u_{n} \right\rVert ^{2} + \dfrac{1}{2 \sigma \beta} \left\lVert u_{n} - u_{n-1} \right\rVert ^{2} - \dfrac{1}{2 \sigma \beta} \left\lVert u_{n-1} \right\rVert ^{2} .
	\end{split}
	\end{align*}
	Therefore, for any $N \geq 1$, we have
	\begin{equation*}
	\mysum_{n=1}^{N} \left( \Psi_{n} - \underline{\Psi} \right) \geq \dfrac{1}{2 \sigma \beta} \mysum_{n=1}^{N} \left\lVert u_{n} - u_{n-1} \right\rVert ^{2} + \dfrac{1}{2 \sigma \beta} \left\lVert u_{N} \right\rVert ^{2} - \dfrac{1}{2 \sigma \beta} \left\lVert u_{0} \right\rVert ^{2} \geq - \dfrac{1}{2 \sigma \beta} \left\lVert u_{0} \right\rVert ^{2} ,
	\end{equation*}
	which leads to a contradiction.
	As $\left\lbrace \Psi_{n} \right\rbrace_{n \geq 1}$ is bounded from below, we obtain from Lemma \ref{lem:conv-pre}  statement \ref{lim:i} and also that
	\begin{equation*}
	x_{n+1} - x_{n} \to 0 , \ y_{n+1} - y_{n} \to 0 \textrm{ and } \ u_{n+1} - u_{n} \to 0 \textrm{ as } \ n \to + \infty.
	\end{equation*}
Since for any $n \geq 1$ it holds
	\begin{align}	
	\begin{split}
	\label{lim:z}
	\left\lVert z_{n+1} - z_{n} \right\rVert & \leq \left\lVert A \right\rVert \left\lVert x_{n+1} - x_{n} \right\rVert + \left\lVert Ax_{n+1} - z_{n+1} \right\rVert + \left\lVert Ax_{n} - z_{n} \right\rVert \\
	& = \left\lVert A \right\rVert \left\lVert x_{n+1} - x_{n} \right\rVert + \dfrac{1}{\sigma \beta} \left\lVert u_{n+1} - u_{n} \right\rVert + \dfrac{1}{\sigma \beta} \left\lVert u_{n} - u_{n-1} \right\rVert,
	\end{split}
	\end{align}
it follows that $z_{n+1} - z_n \to 0$ as $n \to +\infty$.
\end{proof}

\begin{rmk}
Usually, for nonconvex algorithms, the fact that the sequences of differences of consecutive iterates converge to zero is shown by assuming that the generated sequences are bounded (see \cite{Bot-Nguyen,Li-Pong,Yang-Pong-Chen}). In our analysis the only ingredients for obtaining statement (ii) in Theorem \ref{thm:lim} are the descent property and Lemma \ref{lem:conv-pre}. 

As one can notice, the assumption that $\min\{C_2, C_3, C_4\} >0$ plays an essential role in our analysis. In Lemma \ref{lem:pos} we provide possible choices of the algorithm parameters, which lead to the fulfillment of this assumption. 
However, these choices depend on $\ell_{+,1}$, which, at is turn, is defined as being a finite upper bound for the sequence of Lipschitz constants $(\ell_1(y_n))_{n \geq 0}$ (see \eqref{as:sup}). This condition is definitely fulfilled when $\ell_1$ is globally bounded. 
This is for instance the case when $H$ depends only on $x$ and has a Lipschitz continous gradient (see Remark \ref{remark1}(i)), but also when $H$ depends only on $y$.
\end{rmk}

\subsection{General conditions for the boundedness of $\left\lbrace \left( x_{n} , y_{n} , z_{n} , u_{n} \right) \right\rbrace _{n \geq 0}$}

In the following we will formulate general conditions in terms of the input data of the optimization problem \eqref{intro:pb} which guarantee the boundedness of the sequence $\left\lbrace \left( x_{n} , y_{n} , z_{n} , u_{n} \right) \right\rbrace _{n \geq 0}$. Working in the setting of Theorem \ref{thm:lim},
thanks to \eqref{lim:vanish}, we have that the sequences $\left\lbrace x_{n+1} - x_{n} \right\rbrace _{n \geq 0}$, $\left\lbrace y_{n+1} - y_{n} \right\rbrace _{n \geq 0}$, $\left\lbrace z_{n+1} - z_{n} \right\rbrace _{n \geq 0}$ and $\left\lbrace u_{n+1} - u_{n} \right\rbrace _{n \geq 0}$ are bounded.  Denote
	\begin{equation*}
	s_{*} := \sup\limits_{n \geq 0} \left\lbrace \left\lVert x_{n+1} - x_{n} \right\rVert , \left\lVert y_{n+1} - y_{n} \right\rVert , \left\lVert z_{n+1} - z_{n} \right\rVert , \left\lVert u_{n+1} - u_{n} \right\rVert \right\rbrace < + \infty .
	\end{equation*}
Even though this observation does not imply immediately that $\left\lbrace \left( x_{n} , y_{n} , z_{n} , u_{n} \right) \right\rbrace _{n \geq 0}$ is bounded, this will follow under standard coercivity assumptions. Recall that a function $\psi : \sR^{d} \to \sR \cup \left\lbrace + \infty \right\rbrace$ is called coercive, if $\lim_{\left\lVert x \right\rVert \to +\infty} \psi \left( x \right) = + \infty$. 

\begin{thm}
	\label{thm:bound}
Let Assumption \ref{as:bound} be satisfied and the parameters in Algorithm \ref{algo:PALM} be such that $2 \tau \geq \beta \left\lVert A \right\rVert ^{2}$, the constants defined in Lemma \ref{lem:dec} fulfil $\min\{C_2, C_3, C_4\} >0$ and there exist $\gamma_1, \gamma_2 \in \R \setminus \{0\}$ such that \eqref{pos:gamma} holds. Suppose that one of the following conditions hold:
	\begin{enumerate}
		\item the function $H$ is coercive;
		\item the operator $A$ is invertible, and $F$ and $G$ are coercive.		
	\end{enumerate}
Then every sequence $\left\lbrace \left( x_{n} , y_{n} , z_{n} , u_{n} \right) \right\rbrace _{n \geq 0}$ generated by Algorithm \ref{algo:PALM} is bounded.
\end{thm}
\begin{proof}
	Let $n \geq 1$ be fixed. According to Lemma \ref{lem:dec} we have that 
	\begin{align}
\Psi_{1} \geq  & \ldots \geq \Psi_{n} \geq \Psi_{n+1} \nonumber \\
  \geq & \ F \left( z_{n+1} \right) + G \left( y_{n+1} \right) + H \left( x_{n+1} , y_{n+1} \right) - \dfrac{1}{2 \beta} \left\lVert u_{n+1} \right\rVert ^{2} + \dfrac{\beta}{2} \left\lVert Ax_{n+1} - z_{n+1} + \dfrac{1}{\beta} u_{n+1} \right\rVert ^{2}. \label{bound:seq}
	\end{align}
After multiplying \eqref{algo:PALM:x} by $-\tau$ and using \eqref{algo:PALM:u} it yields
	\begin{align}
	A^{T} u_{n+1} & = A^Tu_n + \sigma \beta A^T(Ax_{n+1} - z_{n+1}) = A^Tu_n + (\sigma-1) \beta A^T(Ax_{n+1}-z_{n+1}) + \beta A^T(Ax_{n+1}-z_{n+1})\nonumber \\
& = \left( 1 - \dfrac{1}{\sigma} \right) A^{T} \left( u_{n+1} - u_{n} \right) + A^Tu_n + \beta A^T(Ax_n - z_{n+1}) + \beta A^TA(x_{n+1}-x_n)\nonumber \\
& =  \left( 1 - \dfrac{1}{\sigma} \right) A^{T} \left( u_{n+1} - u_{n} \right) + (\tau \Id - \beta A^TA)(x_{n}-x_{n+1}) - \nabla_{x} H \left( x_{n} , y_{n+1} \right)\nonumber \\
&=
 \left( 1 - \dfrac{1}{\sigma} \right) A^{T} \left( u_{n+1} - u_{n} \right) + B \left( x_{n} - x_{n+1} \right) \nonumber \\
	& \qquad + \nabla_{x} H \left( x_{n+1} , y_{n+1} \right) - \nabla_{x} H \left( x_{n} , y_{n+1} \right) - \nabla_{x} H \left( x_{n+1} , y_{n+1} \right).	\label{bound:u}
	\end{align}
This implies
	\begin{align*}
	\begin{split}
	\left\lVert A^{T} u_{n+1} \right\rVert & \leq \left(\dfrac{1}{\sigma} - 1 \right) \left\lVert A \right\rVert \left\lVert u_{n+1} - u_{n} \right\rVert + \left( \tau + \ell_{1,+} \right) \left\lVert x_{n+1} - x_{n} \right\rVert + \left\lVert \nabla_{x} H \left( x_{n+1} , y_{n+1} \right) \right\rVert \\
	& \leq \left( \left( \dfrac{1}{\sigma} - 1 \right) \left\lVert A \right\rVert + \tau + \ell_{1,+} \right) s_{*} + \left\lVert \nabla_{x} H \left( x_{n+1} , y_{n+1} \right) \right\rVert .
	\end{split}
	\end{align*}
By using the Cauchy-Schwarz inequality we further obtain
	\begin{equation*}
	\lambda_{\min} \left( AA^{T} \right) \left\lVert u_{n+1} \right\rVert ^{2} \leq \left\lVert A^{T} u_{n+1} \right\rVert ^{2} \leq 2 \left( \left(\dfrac{1}{\sigma} - 1 \right) \left\lVert A \right\rVert + \tau + \ell_{1,+} \right) ^{2} s_{*}^{2} + 2 \left\lVert \nabla_{x} H \left( x_{n+1} , y_{n+1} \right) \right\rVert ^{2} .
	\end{equation*}
Multiplying the above relation by $\dfrac{1}{2 \beta \lambda_{\min} \left( AA^{T} \right)}$ and  combining it with \eqref{bound:seq}, we get
	\begin{align}
	\Psi_{1} & \geq F \left( z_{n+1} \right) + G \left( y_{n+1} \right) + H \left( x_{n+1} , y_{n+1} \right) - \dfrac{1}{\beta \lambda_{\min} \left( AA^{T} \right)} \left\lVert \nabla_{x} H \left( x_{n+1} , y_{n+1} \right) \right\rVert ^{2}  \nonumber \\
	& \qquad -  \dfrac{1}{\beta \lambda_{\min} \left( AA^{T} \right)}  \left( \left(\dfrac{1}{\sigma} - 1 \right) \left\lVert A \right\rVert + \tau + \ell_{1,+} \right) ^{2} s_{*}^{2} +  \dfrac{\beta}{2} \left\lVert Ax_{n+1} - z_{n+1} + \dfrac{1}{\beta} u_{n+1} \right\rVert ^{2}. \label{bound:use}
	\end{align}
We will prove the boundedness of $\left\lbrace \left( x_{n} , y_{n} , z_{n} , u_{n} \right) \right\rbrace _{n \geq 0}$ in each of the two scenarios.	
	\begin{enumerate}
		\item 
According to \eqref{bound:use} and Proposition \ref{prop:lips-semiconvex}, we have that for any $n \geq 1$
		\begin{align*}
		& \dfrac{1}{2} H \left( x_{n+1} , y_{n+1} \right) + \dfrac{\beta}{2} \left\lVert Ax_{n+1} - z_{n+1} + \dfrac{1}{\beta} u_{n+1} \right\rVert ^{2} \\
		\leq & \ \Psi_{1} +\dfrac{1}{\beta \lambda_{\min} \left( AA^{T} \right)}  \left( \left( \dfrac{1}{\sigma} - 1 \right) \left\lVert A \right\rVert + \tau + \ell_{1,+} \right) ^{2} s_{*}^{2} - \inf\limits_{z \in \sR^{p}} F \left( z \right) - \inf\limits_{y \in \sR^{m}} G \left( y \right)\\
& \ - \dfrac{1}{2} \inf\limits_{n \geq 1} \left\lbrace H \left( x_{n+1} , y_{n+1} \right) - \left (\dfrac{1}{\gamma_2} - \dfrac{\ell_{1,+}}{2 \gamma_2^2} \right) \left\lVert \nabla_{x} H \left( x_{n+1} , y_{n+1} \right) \right\rVert ^{2} \right\rbrace \\
\leq & \ \Psi_{1} +\dfrac{1}{\beta \lambda_{\min} \left( AA^{T} \right)}  \left( \left(\dfrac{1}{\sigma} - 1 \right) \left\lVert A \right\rVert + \tau + \ell_{1,+} \right) ^{2} s_{*}^{2} - \inf\limits_{z \in \sR^{p}} F \left( z \right) - \inf\limits_{y \in \sR^{q}} G \left( y \right) - \inf\limits_{(x,y) \in \sR^{m} \times \R^q} H \left(x, y \right)\\
 < & \ + \infty.
		\end{align*}
Since $H$ is coercive and bounded from below, it follows that $\left\lbrace \left( x_{n} , y_{n} \right) \right\rbrace _{n \geq 0}$ and 
		$\left\lbrace Ax_{n} - z_{n} + \dfrac{1}{\beta} u_{n} \right\rbrace_{n \geq 0}$ are bounded. 
As, according to \eqref{algo:PALM:u}, $\left\lbrace Ax_{n} - z_{n} \right\rbrace _{n \geq 0}$ is bounded, it follows that $\left\lbrace u_{n} \right\rbrace _{n \geq 0}$ and $\left\lbrace z_{n} \right\rbrace _{n \geq 0}$ are also bounded.
		
\item According to \eqref{bound:use} and Proposition \ref{prop:lips-semiconvex}, we have this time that for any $n \geq 1$
		\begin{align*}
		&  \ F \left( z_{n+1} \right) + G \left( y_{n+1} \right) + \dfrac{\beta}{2} \left\lVert Ax_{n+1} - z_{n+1} + \dfrac{1}{\beta} u_{n+1} \right\rVert ^{2} \\
		\leq & \  \Psi_{1} + \dfrac{1}{\beta \lambda_{\min} \left( AA^{T} \right)}  \left( \left( \dfrac{1}{\sigma} - 1 \right) \left\lVert A \right\rVert + \tau + \ell_{1,+} \right) ^{2} s_{*}^{2}\\
&  \ - \inf_{n \geq 1} \left\lbrace H \left( x_{n+1} , y_{n+1} \right) -\left (\dfrac{1}{\gamma_1} - \dfrac{\ell_{1,+}}{2 \gamma_1^2} \right) \left\lVert \nabla_{x} H \left( x_{n+1} , y_{n+1} \right) \right\rVert ^{2} \right\rbrace\\
\leq &   \  \Psi_{1} + \dfrac{1}{\beta \lambda_{\min} \left( AA^{T} \right)}  \left( \left(\dfrac{1}{\sigma} - 1 \right) \left\lVert A \right\rVert + \tau + \ell_{1,+} \right) ^{2} s_{*}^{2} - \inf\limits_{(x,y) \in \sR^{m} \times \R^q} H \left(x, y \right) < + \infty. 
		\end{align*}
		Since $F$ and $G$ are coercive and bounded from below, it follows that the sequences $\left\lbrace \left( y_{n} , z_{n} \right) \right\rbrace _{n \geq 0}$ and $\left\lbrace Ax_{n} - z_{n} + \dfrac{1}{\beta} u_{n} \right\rbrace_{n \geq 0}$ are bounded. As, according to \eqref{algo:PALM:u}, $\left\lbrace Ax_{n} - z_{n} \right\rbrace _{n \geq 0}$ is bounded, it follows that $\left\lbrace u_{n} \right\rbrace _{n \geq 0}$ and $\left\lbrace Ax_{n} \right\rbrace _{n \geq 0}$ are bounded. The fact that $A$ is invertible implies that $\left\lbrace x_{n} \right\rbrace _{n \geq 0}$ is bounded.
		\qedhere
	\end{enumerate}	
\end{proof}

\subsection{The cluster points of $\left\lbrace \left( x_{n} , y_{n} , z_{n} , u_{n} \right) \right\rbrace _{n \geq 0}$ are KKT points}

We will close this section dedicated to the convergence analysis of the sequence generated by  Algorithm \ref{algo:PALM} in a general framework by proving that any cluster point of $\left\lbrace \left( x_{n} , y_{n} , z_{n} , u_{n} \right) \right\rbrace _{n \geq 0}$ is a KKT point of the optimization problem \eqref{intro:pb}. We provided above general conditions which guarantee both the descent inequality \eqref{dec:inq}, with positive constants  $\Cprex, \Cprey$ and $\Cpreu$, and the boundedness of the generated iterates. Lemma \ref{lem:pos} and Theorem \ref{thm:bound} provide one possible setting that ensures these two fundamental properties of the convergence analysis. We do not want to restrict ourselves to this particular setting and, therefore, we will work, from now on, under the following assumptions.

\begin{assume}
	\label{as}
	\begin{enumerate}
		\item
		\label{as:i}
		the functions $F, G$ and $H$ are bounded from below;

\item 
		\label{as:iii}
		the linear operator $A$ is surjective;

\item
		\label{as:iv}
		every sequence $\left\lbrace \left( x_{n} , y_{n} , z_{n} , u_{n} \right) \right\rbrace _{n \geq 0}$ generated by the Algorithm \ref{algo:PALM} is bounded:
		
		\item
		\label{as:ii}
		$\nabla H$ is Lipschitz continuous with constant $L>0$ on a convex bounded subset $B_{1} \times B_{2} \subseteq \sR^{m} \times \sR^{q}$ containing $\left\lbrace \left( x_{n} , y_{n}  \right) \right\rbrace _{n \geq 0}$. In other words, for any $\left( x , y \right), \left( x' , y' \right) \in B_{1} \times B_{2}$ it holds
		\begin{equation}
		\label{as:lip}
		\opnorm{\left( \nabla_{x} H \left( x , y \right) - \nabla_{x} H \left( x' , y' \right) , \nabla_{y} H \left( x , y \right) - \nabla_{y} H \left( x' , y' \right) \right)} \leq L \opnorm{\left( x , y \right) - \left( x' , y' \right)} ;
		\end{equation}
		
		\item
		\label{as:v} the parameters $\mu, \beta, \tau >0$ and $0 < \sigma \leq 1$ are such that $2 \tau \geq \beta \|A\|^2$ and $\min\{C_2, C_3, C_4\} >0$, where
		\begin{align*}
		\begin{split}
		\Cdecx 	& := \tau - \dfrac{L \sqrt{2} + \beta \left\lVert A \right\rVert ^{2}}{2} - \dfrac{4 \sigma \tau^{2}}{\beta \lambda_{\min} \left( AA^{T} \right)} - \dfrac{8 \left( \sigma \tau + L \sqrt{2} \right) ^{2}}{\sigma \beta \lambda_{\min} \left( AA^{T} \right)} ,		
		\end{split}
		\\
		\begin{split}
		\Cdecy	& := \dfrac{\mu - L \sqrt{2}}{2} - \dfrac{16L^{2}}{\sigma \beta \lambda_{\min} \left( AA^{T} \right)} ,
		\end{split}
		\\
		\begin{split}
		\Cdecu 	& := \dfrac{1}{\sigma \beta}.
		\end{split}
		\end{align*}
	\end{enumerate}
\end{assume}

\begin{rmk}
Being facilitated by the boundedness of the generated sequence, Assumption \ref{as} \ref{as:ii} 
not only guarantee the fulfilment of Assumption \ref{as:bound} \ref{as:bound:ii} and \ref{as:bound:iii} on a convex bounded set, but it also arises in a more natural way (see also \cite{Bolte-Sabach-Teboulle:MP}). Assumption \ref{as} \ref{as:ii} holds, for instance, if $H$ is twice continuously differentiable. In addition, as \eqref{as:lip} implies for any $\left( x , y \right), \left( x' , y' \right) \in B_{1} \times B_{2}$ that
	\begin{equation*}
	\left\lVert \nabla_{x} H \left( x , y \right) - \nabla_{x} H \left( x' , y' \right) \right\rVert + \left\lVert \nabla_{y} H \left( x , y \right) - \nabla_{y} H \left( x' , y' \right) \right\rVert \leq L \sqrt{2} \left( \left\lVert x - x' \right\rVert + \left\lVert y - y' \right\rVert \right),
	\end{equation*}
we can take
	\begin{equation}
	\label{const:new-lip}
	\ell_{1,+} = \ell_{2,+} = \ell_{3,+} := L \sqrt{2} .
	\end{equation}
As \eqref{as:lip-xx} - \eqref{as:lip-xy} are valid also on a convex bounded set, the descent inequality 
	\begin{equation}
	\label{dec:new-inq}
	\Psi_{n+1} + \Cdecx \left\lVert x_{n+1} - x_{n} \right\rVert ^{2} + \Cdecy \left\lVert y_{n+1} - y_{n} \right\rVert ^{2} + \Cdecu \left\lVert u_{n+1} - u_{n} \right\rVert ^{2} \leq \Psi_{n} \ \forall n \geq 1
	\end{equation}
remains true, for constants $\Cprex, \Cprey, \Cpreu$  taken as in Lemma \ref{lem:dec} and by taking into consideration \eqref{const:new-lip}. A possible choice of the parameters of the algorithm such that $\min \left\lbrace \Cprex , \Cprey , \Cpreu \right\rbrace > 0$ can be obtained also from Lemma \ref{lem:pos}.
\end{rmk}

The next result provide upper estimates for the limiting subgradients of the regularized function $\Psi$ at $\left( x_{n} , y_{n} , z_{n} , u_{n} \right)$ for every $n \geq 1$.
\begin{lem}
	\label{lem:sgra}
	Let Assumption \ref{as} be satisfied and $\left\lbrace \left( x_{n} , y_{n} , z_{n} , u_{n} \right) \right\rbrace _{n \geq 0}$ be a sequence generated by Algorithm \ref{algo:PALM}. Then for any $n \geq 1$ it holds
	\begin{equation}
	\label{sgr:Psi}
	D_{n} := \left( d_{x}^{n} , d_{y}^{n} , d_{z}^{n} , d_{u}^{n} , d_{x'}^{n} , d_{u'}^{n} \right) \in \partial \Psi \left( \X_{n} \right) ,	
	\end{equation}
	where
	\begin{subequations}
		\label{sgr:all}
		\begin{align}
		\begin{split}
		\label{sgr:x}
		d_{x}^{n} 	& := \nabla_{x} H \left( x_{n} , y_{n} \right) + A^{T} u_{n} + \beta A^{T} \left( Ax_{n} - z_{n} \right) + 2 \CPsix \left( x_{n} - x_{n-1} \right) \\
		& \qquad + 2\sigma \CPsiu B^{T} \left( A^{T} \left( u_{n} - u_{n-1} \right) + \sigma B \left( x_{n} - x_{n-1} \right) \right) ,
		\end{split}
		\\
		\begin{split}
		\label{sgr:y}
		d_{y}^{n} 	& := \nabla_{y} H \left( x_{n} , y_{n} \right) - \nabla_{y} H \left( x_{n-1} , y_{n-1} \right) + \mu \left( y_{n-1} - y_{n} \right) , 
		\end{split}
		\\
		\begin{split}
		\label{sgr:z}
		d_{z}^{n} 	& := u_{n-1} - u_{n} + \beta A \left( x_{n-1} - x_{n} \right) ,
		\end{split}		
		\\
		\begin{split}
		\label{sgr:u}
		d_{u}^{n} 	& := Ax_{n} - z_{n} + 2 \CPsiu A \left( A^{T} \left( u_{n} - u_{n-1} \right) + \sigma B \left( x_{n} - x_{n-1} \right) \right) , 
		\end{split}
		\\
		\begin{split}
		\label{sgr:x'}
		d_{x'}^{n} 	& := - 2\sigma \CPsiu B^{T} \left( A^{T} \left( u_{n} - u_{n-1} \right) + \sigma B \left( x_{n} - x_{n-1} \right) \right) - 2 \CPsix \left( x_{n} - x_{n-1} \right) , 
		\end{split}
		\\
		\begin{split}
		\label{sgr:u'}
		d_{u'}^{n} 	& := - 2 \CPsiu A \left( A^{T} \left( u_{n} - u_{n-1} \right) + \sigma B \left( x_{n} - x_{n-1} \right) \right) .
		\end{split}
		\end{align}
	\end{subequations}
In addition, for any $n \geq 1$ it holds
\begin{equation}
\label{sgr:inq}
\opnorm{D_{n}} \leq \Csgrx \left\lVert x_{n} - x_{n-1} \right\rVert + \Csgry \left\lVert y_{n} - y_{n-1} \right\rVert + \Csgru \left\lVert u_{n} - u_{n-1} \right\rVert,
\end{equation}
where
	\begin{align*}
	\begin{split}
	\Csgrx 	& := 2 \sqrt{2} \cdot L + \tau + \beta \left\lVert A \right\rVert + 4 \left( \sigma \tau + \left\lVert A \right\rVert \right) \sigma \tau \CPsiu + 4 \CPsix , 
	\end{split}
	\\
	\begin{split}	
	\Csgry 	& := L \sqrt{2} + \mu ,
	\end{split}
	\\
	\begin{split}
	\Csgru 	& := 1 + \dfrac{1}{\sigma \beta} + \left( \dfrac{2}{\sigma} - 1 \right) \left\lVert A \right\rVert + 4 \left( \sigma \tau + \left\lVert A \right\rVert \right) \CPsiu \left\lVert A \right\rVert .
	\end{split}
	\end{align*}
\end{lem}
\begin{proof}
Let $n \geq 1$ be fixed. Applying the calculus rules of the limiting subdifferential we get
	\begin{subequations}
		\label{sgr:Psi-all}
		\begin{align}
		\begin{split}
		\label{sgr:Psi-x}
		\nabla_{x} \Psi \left( \X_{n} \right) 		& = \nabla_{x} H \left( x_{n} , y_{n} \right) + A^{T} u_{n} + \beta A^{T} \left( Ax_{n} - z_{n} \right) + 2 \CPsix \left( x_{n} - x_{n-1} \right) \\
		& \qquad + 2 \sigma \CPsiu B^{T} \left( A^{T} \left( u_{n} - u_{n-1} \right) + \sigma B \left( x_{n} - x_{n-1} \right) \right) ,
		\end{split}		
		\\
		\begin{split}
		\label{sgr:Psi-y}
		\partial_{y} \Psi \left( \X_{n} \right) 	& = \partial G \left( y_{n} \right) + \nabla_{y} H \left( x_{n} , y_{n} \right) , 
		\end{split}
		\\
		\begin{split}
		\label{sgr:Psi-z}
		\partial_{z} \Psi \left( \X_{n} \right) 	& = \partial F \left( z_{n} \right) - u_{n} - \beta \left( Ax_{n} - z_{n} \right) ,
		\end{split}
		\\		
		\begin{split}
		\label{sgr:Psi-u}
		\nabla_{u} \Psi \left( \X_{n} \right) 		& = Ax_{n} - z_{n} + 2 \CPsiu A \left( A^{T} \left( u_{n} - u_{n-1} \right) + \sigma B \left( x_{n} - x_{n-1} \right) \right) , 
		\end{split}
		\\
		\begin{split}
		\label{sgr:Psi-x'}
		\nabla_{x'} \Psi \left( \X_{n} \right) 	& = - 2 \sigma \CPsiu B^{T} \left( A^{T} \left( u_{n} - u_{n-1} \right) + \sigma B \left( x_{n} - x_{n-1} \right) \right) - 2 \CPsix \left( x_{n} - x_{n-1} \right) , 
		\end{split}
		\\
		\begin{split}
		\label{sgr:Psi-u'}
		\nabla_{u'} \Psi \left( \X_{n} \right) 	& = - 2 \CPsiu A \left( A^{T} \left( u_{n} - u_{n-1} \right) + \sigma B \left( x_{n} - x_{n-1} \right) \right) .
		\end{split}
		\end{align}
	\end{subequations}
Then \eqref{sgr:x} and \eqref{sgr:u} - \eqref{sgr:u'} follow directly from \eqref{sgr:Psi-x} and \eqref{sgr:Psi-u} - \eqref{sgr:Psi-u'}, respectively.
By combining \eqref{sgr:Psi-y} with the optimality criterion for \eqref{algo:PALM:y}
\begin{equation*}
0 \in \partial G \left( y_{n} \right) + \nabla_{y} H \left( x_{n-1} , y_{n-1} \right) + \mu \left( y_{n} - y_{n-1} \right) ,
\end{equation*}
we obtain \eqref{sgr:y}.
Similarly, by combining \eqref{sgr:Psi-z} with the optimality criterion for \eqref{algo:PALM:z}
\begin{equation*}
0 \in \partial F \left( z_{n} \right) - u_{n-1} - \beta \left( Ax_{n-1} - z_{n} \right),
\end{equation*}
we get \eqref{sgr:z}.

In the following we will derive the upper estimates for the components of the limiting subgradient.
From \eqref{bound:u} it follows
\begin{align*}
\left\lVert d_{x}^{n} \right\rVert \leq & \left\lVert \nabla_{x} H \left( x_{n} , y_{n} \right) + A^{T} u_{n} \right\rVert + \beta \left\lVert A \right\rVert \left\lVert Ax_{n} - z_{n} \right\rVert + 2 \left( \CPsix + \sigma^{2} \tau^{2} \CPsiu \right) \left\lVert x_{n} - x_{n-1} \right\rVert \\
&  + 2 \sigma \tau \CPsiu \left\lVert A \right\rVert \left\lVert u_{n} - u_{n-1} \right\rVert \\
\leq & \left( L \sqrt{2} + \tau + 2 \CPsix + 2 \sigma^{2} \tau^{2} \CPsiu \right) \left\lVert x_{n} - x_{n-1} \right\rVert + \left( \dfrac{2}{\sigma} - 1 + 2 \sigma \tau \CPsiu \right) \left\lVert A \right\rVert \left\lVert u_{n} - u_{n-1} \right\rVert.
\end{align*}
In addition, we have
\begin{align*}
\begin{split}
\left\lVert d_{y}^{n} \right\rVert 	& \leq L \sqrt{2} \left\lVert x_{n} - x_{n-1} \right\rVert + \left( L \sqrt{2} + \mu \right) \left\lVert y_{n} - y_{n-1} \right\rVert , \\
\left\lVert d_{z}^{n} \right\rVert 	& \leq \beta \left\lVert A \right\rVert \left\lVert x_{n} - x_{n-1} \right\rVert + \left\lVert u_{n} - u_{n-1} \right\rVert , \\
\left\lVert d_{u}^{n} \right\rVert 	& \leq 2 \sigma \tau \CPsiu \left\lVert A \right\rVert \left\lVert x_{n} - x_{n-1} \right\rVert + \left( \dfrac{1}{\sigma \beta} + 2 \CPsiu \left\lVert A \right\rVert ^{2} \right) \left\lVert u_{n} - u_{n-1} \right\rVert , \\
\left\lVert d_{x'}^{n} \right\rVert & \leq 2 \left( \sigma^{2} \tau^{2} \CPsiu + \CPsix \right) \left\lVert x_{n} - x_{n-1} \right\rVert + 2 \sigma \tau \CPsiu \left\lVert A \right\rVert \left\lVert u_{n} - u_{n-1} \right\rVert , \\
\left\lVert d_{u'}^{n} \right\rVert & \leq 2 \sigma \tau \CPsiu \left\lVert A \right\rVert \left\lVert x_{n} - x_{n-1} \right\rVert + 2 \CPsiu \left\lVert A \right\rVert ^{2} \left\lVert u_{n} - u_{n-1} \right\rVert .
\end{split}
\end{align*}
The inequality \eqref{sgr:inq} follows by combining the above relations with \eqref{intro:norm-inq}.
\end{proof}

We denote by $\Omega := \Omega \left( \left\lbrace \X_{n} \right\rbrace _{n \geq 1} \right)$ the set of cluster points of the sequence $\left\lbrace \X_{n} \right\rbrace _{n \geq 1} \subseteq \bR$, which is nonempty thanks to the boundedness of $\left\lbrace \X_{n} \right\rbrace _{n \geq 1}$.
The distance function of the set $\Omega$ is defined for any $\X \in \bR$ by $\dist \left( \X , \Omega \right) := \inf \left\lbrace \opnorm{ \X - \Y } \colon \Y \in \Omega \right\rbrace$. The main result of this section follows.
\begin{thm}
	\label{lem:clus}
	Let Assumption \ref{as} be satisfied and $\left\lbrace \left( x_{n} , y_{n} , z_{n} , u_{n} \right) \right\rbrace _{n \geq 0}$ be a sequence generated by Algorithm \ref{algo:PALM}.  The following statements are true:
	\begin{enumerate}
		\item
		\label{lem:clus:i}
		if $\left\lbrace \left( x_{n_{k}} , y_{n_{k}} , z_{n_{k}} , u_{n_{k}} \right) \right\rbrace _{k \geq 0}$ is a subsequence of $\left\lbrace \left( x_{n} , y_{n} , z_{n} , u_{n} \right) \right\rbrace _{n \geq 0}$ which converges to $\left( x_{*} , y_{*} , z_{*} , u_{*} \right)$  as $k \to +\infty$, then
		\begin{equation*}
		\lim\limits_{k \to + \infty} \Psi_{n_{k}} = \Psi \left( x_{*} , y_{*} , z_{*} , u_{*} , x_{*} , u_{*} \right);
		\end{equation*}	
		\item 
		\label{lem:clus:ii}
		it holds
\begin{align}\label{clus:KKT}
\Omega & \subseteq \crit \left( \Psi \right) \nonumber \\
& \subseteq \{\X_{*} \in \bR: - A^{T} u_{*} = \nabla_{x} H \left( x_{*} , y_{*} \right), 0 \in \partial G \left( y_{*} \right) + \nabla_{y} H \left( x_{*} , y_{*} \right), u_{*} \in \partial F \left( z_{*} \right) , z_{*} = A x_{*}\},
\end{align}
where $\X_{*} := \left( x_{*}, y_{*}, z_{*}, u_{*} , x_{*}, u_{*} \right)$;
		\item
		\label{lem:clus:iii}
		it holds $\lim\limits_{n \to +\infty} \dist \left( \X_{n} , \Omega \right) = 0$;
		\item
		\label{lem:clus:iv}
		the set $\Omega$ is nonempty, connected and compact;		
		\item
		\label{lem:clus:v}
		the function $\Psi$ takes on $\Omega$ the value $\Psi_{*} 
		= \lim\limits_{n \to +\infty} \Psi_{n}=\lim\limits_{n \to +\infty} \left\{F \left( z_{n} \right) + G \left( y_{n} \right) + H \left( x_{n} , y_{n}\right)\right\}$.
	\end{enumerate}
\end{thm}
\begin{proof}
	Let $\left( x_{*} , y_{*} , z_{*} , u_{*} \right) \in \sR^{m} \times \sR^{q} \times \sR^{p} \times \sR^{p}$ be such that the subsequence
	\begin{equation*}
	\left\lbrace \X_{n_{k}} := \left( x_{n_{k}} , y_{n_{k}} , z_{n_{k}} , u_{n_{k}} , x_{n_{k}-1} , u_{n_{k}-1} \right) \right\rbrace _{k \geq 1}
	\end{equation*}
of $\left\lbrace \X_{n} \right\rbrace _{n \geq 1}$ converges to $\X_{*} := \left( x_{*} , y_{*} , z_{*} , u_{*} , x_{*} , u_{*} \right)$.
	
	\item[(i)]
	From \eqref{algo:PALM:y} and \eqref{algo:PALM:z} we have for any $k \geq 1$
	\begin{align*}
	& G \left( y_{n_{k}} \right) + \left\langle \nabla_{y} H \left( x_{n_{k}-1} , y_{n_{k}-1} \right) , y_{n_{k}} - y_{n_{k}-1} \right\rangle + \dfrac{\mu}{2} \left\lVert y_{n_{k}} - y_{n_{k}-1} \right\rVert ^{2} \\
	\leq & \ G \left( y_{*} \right) + \left\langle \nabla_{y} H \left( x_{n_{k}-1} , y_{n_{k}-1} \right) , y_{*} - y_{n_{k}-1} \right\rangle + \dfrac{\mu}{2} \left\lVert y_{*} - y_{n_{k}-1} \right\rVert ^{2}
	\end{align*}
	and
	\begin{align*}
	& F \left( z_{n_{k}} \right) + \left\langle u_{n_{k}-1} , Ax_{n_{k}-1} - z_{n_{k}} \right\rangle + \dfrac{\beta}{2} \left\lVert Ax_{n_{k}-1} - z_{n_{k}} \right\rVert ^{2} \\
	\leq & \ F \left( z_{*} \right) + \left\langle u_{n_{k}-1} , Ax_{n_{k}-1} - z_{*} \right\rangle + \dfrac{\beta}{2} \left\lVert Ax_{n_{k-1}} - z_{*} \right\rVert ^{2} ,
	\end{align*}
	respectively. From \eqref{algo:PALM:u} and Theorem \ref{thm:lim} follows $Ax^*=z^*$. 
	Taking the limit superior as $k \to +\infty$ on both sides of the above inequalities, we get
	\begin{equation*}
	\limsup\limits_{k \to + \infty} F \left( z_{n_{k}} \right) \leq F \left( z_{*} \right) \qquad \textrm{ and } \qquad \limsup\limits_{k \to + \infty} G \left( y_{n_{k}} \right) \leq G \left( y_{*} \right)
	\end{equation*}
	which, combined with the lower semicontinuity of $F$ and $G$, lead to
	\begin{equation*}
	\lim\limits_{k \to + \infty} F \left( z_{n_{k}} \right) = F \left( z_{*} \right) \qquad \textrm{ and } \qquad \lim\limits_{k \to + \infty} G \left( y_{n_{k}} \right) = G \left( y_{*} \right) .
	\end{equation*}
	The desired statement follows thanks to the continuity of $H$.
	
	\item[(ii)]
	For the sequence $\left\lbrace D_{n} \right\rbrace _{n \geq 0}$ defined in \eqref{sgr:Psi} - \eqref{sgr:all}, we have that $D_{n_{k}} \in \partial \Psi \left( \X_{n_{k}} \right)$ for any $k \geq 1$ and $D_{n_{k}}  \to 0$ as $k \to +\infty$, while $\X_{n_{k}} \to \X_{*}$ and $\Psi_{n_{k}} \to \Psi ( \X_{*} )$ as $k \to +\infty$. The closedness criterion of the limiting subdifferential guarantees that $0 \in \partial \Psi ( \X_{*} )$ or, in other words, $\X_{*} \in \crit \left( \Psi \right)$.
	
	Choosing now an element $\X_{*} \in \crit \left( \Psi \right)$, it holds
	\begin{equation*}
	\begin{cases}
	0 	& = \nabla_{x} H \left( x_{*}, y_{*} \right) + A^{T} u_{*} + \beta A^{T} \left( A x_{*} - z_{*} \right) , \\
	0 	& \in \partial G \left( y_{*} \right) + \nabla_{y} H \left( x_{*}, y_{*} \right) , \\
	0 	& \in \partial F \left( z_{*} \right) - u_{*} - \beta \left( A x_{*} - z_{*} \right) , \\
	0 	& = A x_{*} - z_{*} , \\
	\end{cases}
	\end{equation*}
	which is further equivalent to \eqref{clus:KKT}.
	
	\item[(iii)-(iv)]
	The proof follows in the lines of the proof of Theorem 5 (ii)-(iii) in \cite{Bolte-Sabach-Teboulle:MP}, also by taking into consideration \cite[Remark 5]{Bolte-Sabach-Teboulle:MP}, according to which the properties in (iii) and (iv) are generic for sequences satisfying $\X_{n} - \X_{n-1} \to 0$ as $n \to +\infty$, which is indeed the case due to \eqref{lim:vanish}.
	
	\item[(v)]
Due to \eqref{lim:vanish} and the fact that $\left\lbrace u_{n} \right\rbrace _{n \geq 0}$ is bounded, the sequences $\left\lbrace F \left( z_{n} \right) + G \left( y_{n} \right) + H \left( x_{n} , y_{n} \right) \right\rbrace _{n \geq 0}$ and $\left\lbrace \Psi_{n} \right\rbrace _{n \geq 0}$ have the same limit
	\begin{equation*}
	\Psi_{*} = \lim\limits_{n \to +\infty} \Psi_{n} = \lim\limits_{n \to +\infty} \left\lbrace F \left( z_{n} \right) + G \left( y_{n} \right) + H \left( x_{n} , y_{n} \right) \right\rbrace .
	\end{equation*}
	The conclusion follows by taking into consideration the first two statements of this theorem.
\end{proof}

\begin{rmk}
	An element $\left( x_{*}, y_{*}, z_{*}, u_{*} \right)$ fulfilling \eqref{clus:KKT}
	is a so-called KKT point of the optimization problem \eqref{intro:pb}. Such a KKT point obviously fulfils 
	\begin{equation}
	\label{KKT:conv-sub1}
	0 \in A^{T} \partial F \left( A x_{*} \right) + \nabla_{x} H \left( x_{*} , y_{*} \right) , \qquad 0 \in \partial G \left( y_{*} \right) + \nabla_{y} H \left( x_{*} , y_{*} \right) .
	\end{equation}
If $A$ is injective, then this system of inclusions is further equivalent to
	\begin{align}
	\label{KKT:conv-sub2}
	0 & \in \partial \left( F \circ A \right) \left( x_{*} \right) + \nabla_{x} H \left( x_{*} , y_{*} \right) = \partial_{x} \left( F \circ A + H \right),  \nonumber\\
	0 & \in \partial G \left( y_{*} \right) + \nabla_{y} H \left( x_{*} , y_{*} \right) = \partial_{y} \left( G + H \right) ,
	\end{align}
	in other words, $\left( x_{*} , y_{*} \right)$ is a critical point of the optimization problem \eqref{intro:pb}.
On the other hand, if the functions $F, G$ and $H$ are convex, then, even without asking $A$ to be injective, \eqref{KKT:conv-sub1} and \eqref{KKT:conv-sub2} are equivalent, which means that $\left( x_{*} , y_{*} \right)$ is a global minimum of the optimization problem \eqref{intro:pb}.
\end{rmk}

\section{Global convergence and rates}
\label{subsec:conv-KL}

In this section we will prove global convergence for the sequence $\left\lbrace \left( x_{n} , y_{n} , z_{n} , u_{n} \right) \right\rbrace _{n \geq 0}$ generated by Algorithm \ref{algo:PALM} in the context of the Kurdyka-\Loja property and provide convergence rates for it in the context of the \Loja property.

\subsection{Global convergence under Kurdyka-\L ojasiewicz assumptions}

The origins of this notion go back to the pioneering work of Kurdyka who introduced in \cite{Kurdyka} a general form of the \Loja inequality \cite{Lojasiewicz}. An  extension to the nonsmooth setting has been proposed and studied in \cite{Bolte-Daniilidis-Lewis, Bolte-Daniilidis-Lewis-Shiota, Bolte-Daniilidis-Ley-Mazet}.

\begin{defi}
	\label{defi:Class-Phi}
	Let $\eta \in \left( 0 , + \infty \right]$. We denote by $\Phi_{\eta}$ the set of all concave and continuous functions $\varphi \colon \left[ 0 , \eta \right) \to [0,+\infty)$ which satisfy the following conditions:
	\begin{enumerate}
		\item $\varphi \left( 0 \right) = 0$;
		\item $\varphi$ is $\sC^{1}$ on $\left( 0 , \eta \right)$ and continuous at $0$;
		\item for any $s \in \left( 0 , \eta \right): \varphi ' \left( s \right) > 0$.
	\end{enumerate}
\end{defi}
\begin{defi}	
	\label{defi:KL}
	Let $\Psi \colon \sR^{d} \to \sR \cup \left\lbrace + \infty \right\rbrace$ be proper and lower 
	semicontinuous.
	\begin{enumerate}
		\item The function $\Psi$ is said to have the Kurdyka-\Loja (\KL) property at a point $\widehat{v} \in \dom \partial \Psi := \left\lbrace v \in \sR^{d} \colon \partial \Psi \left( v \right) \neq \emptyset \right\rbrace$, if there exists $\eta \in \left( 0 , + \infty \right]$, a neighborhood $V$ of $\widehat{v}$ and a function $\varphi \in \Phi_{\eta}$ such that for any
		\begin{equation*}
		\label{intro:intersection}
		v \in V \cap \left\{v \in \sR^d: \Psi \left( \widehat{v} \right) < \Psi \left( v\right) < \Psi \left( \widehat{v} \right) + \eta \right\} 
		\end{equation*}
		the following inequality holds
		\begin{equation*}
		\label{intro:KL-inequality-0}
		\varphi ' \left( \Psi \left( v \right) - \Psi \left( \widehat{v} \right) \right) \cdot \dist \left( \0 , \partial \Psi \left( v \right) \right) \geq 1 .
		\end{equation*}		
		\item If $\Psi$ satisfies the \KL \ property at each point of $\dom \partial \Psi$, then $\Psi$ is called  \KL \ function.
	\end{enumerate}
\end{defi}
The functions $\varphi$ belonging to the set $\Phi_{\eta}$ for $\eta \in \left(0 , + \infty \right]$ are called desingularization functions. The \KL \ property reveals the possibility to reparametrize the values of $\Psi$ in order to avoid flatness around the critical points. To the class of \KL \ functions belong semialgebraic, real subanalytic, uniformly convex functions and convex functions satisfying a growth condition. We refer to \cite{Attouch-Bolte,Attouch-Bolte-Redont-Soubeyran, Attouch-Bolte-Svaiter, Bolte-Daniilidis-Lewis, Bolte-Daniilidis-Lewis-Shiota, Bolte-Daniilidis-Ley-Mazet, Bolte-Sabach-Teboulle:MP} for more properties of \KL \ functions and illustrating examples.

The following result, the proof of which can be found in \cite[Lemma 6]{Bolte-Sabach-Teboulle:MP}, will play an essential role in our convergence analysis.

\begin{lem} \textbf{(Uniformized \KL \ property)}
	\label{lem:uniformized}
	Let $\Omega$ be a compact set and $\Psi \colon \sR^{d} \to \sR \cup \left\lbrace + \infty \right\rbrace$ be a proper and lower semicontinuous function. Assume that $\Psi$ is constant on $\Omega$ and satisfies the \KL \ property at each point of $\Omega$. Then there exist $\varepsilon > 0, \eta > 0$ 
	and $\varphi \in \Phi_{\eta}$ such that for any $\widehat{v} \in \Omega$ and  every element $u$ in the intersection
	\begin{equation*}
	\label{intro:intersection-all}
	\left\lbrace v \in \sR^{d} \colon \dist \left( v , \Omega \right) < \varepsilon \right\rbrace \cap  \left\{v \in \sR^d: \Psi \left( \widehat{v} \right) < \Psi \left( v\right) < \Psi \left( \widehat{v} \right) + \eta \right\} 
	\end{equation*}
	it holds
	\begin{equation*}
	\label{intro:KL-inequality}
	\varphi ' \left( \Psi \left( v \right) - \Psi \left( \widehat{v} \right) \right) \cdot \dist \left( \0 , \partial \Psi \left( v \right) \right) \geq 1 .
	\end{equation*}
\end{lem}

From now on we will use the following notations
\begin{equation*}
\Cmin := \dfrac{1}{\min \left\lbrace \Cdecx , \Cdecy , \Cdecu \right\rbrace} , \qquad \Cmax := \max \left\lbrace \Csgrx , \Csgry , \Csgru \right\rbrace \qquad \textrm{ and } \qquad \E_{n} := \Psi_{n} - \Psi_{*} \ \forall n \geq 1,
\end{equation*}
where $\Psi_{*} = \lim\limits_{n \to +\infty} \Psi_{n}$. 

The next result shows that if $\Psi$ is a \KL \ function, then the sequence 
$\left\lbrace \left( x_{n} , y_{n} , z_{n} , u_{n} \right) \right\rbrace _{n \geq 0}$ converges to a KKT point of the optimization 
problem \eqref{intro:pb}. This hypothesis is fulfilled if, for instance, $F, G$ and $H$ are semi-algebraic functions.

\begin{thm}
	\label{thm:conv}
	Let Assumption \ref{as} be satisfied and $\left\lbrace \left( x_{n} , y_{n} , z_{n} , u_{n} \right) \right\rbrace _{n \geq 0}$ be a sequence generated by Algorithm \ref{algo:PALM}.  If $\Psi$ is a \KL \ function, then the following statements are true: 
	\begin{enumerate}
		\item the sequence $\left\lbrace \left( x_{n} , y_{n} , z_{n} , u_{n} \right) \right\rbrace _{n \geq 0}$ has finite length, namely,
		\begin{equation}
		\label{conv:Cauchy}
		\mysum_{n \geq 0} \left\lVert x_{n+1} - x_{n} \right\rVert < + \infty , \mysum_{n \geq 0} \left\lVert y_{n+1} - y_{n} \right\rVert < + \infty , \mysum_{n \geq 0} \left\lVert z_{n+1} - z_{n} \right\rVert < + \infty , \mysum_{n \geq 0} \left\lVert u_{n+1} - u_{n} \right\rVert < + \infty;
		\end{equation}
		\item the sequence$\left\lbrace \left( x_{n} , y_{n} , z_{n} , u_{n} \right) \right\rbrace _{n \geq 0}$ converges to a KKT point of the optimization problem \eqref{intro:pb}.
	\end{enumerate}
\end{thm}
\begin{proof}
	Let be $\X_{*} \in \Omega$, thus $\Psi \left( \X_{*} \right) = \Psi_{*}$. 
	Recall that $\left\lbrace \E_{n} \right\rbrace _{n \geq 1}$ is monotonically decreasing and converges to $0$ as $n \to + \infty$. We consider two cases.
	
	\item[Case 1.] Assume that there exists an integer $n' \geq 1$ such that $\E_{n'} = 0$ or, equivalently, $\Psi_{n'} = \Psi_{*}$. Due to the monotonicity of $\left\lbrace \E_{n} \right\rbrace _{n \geq 1}$, it follows that $\E_{n} = 0$ or, equivalently, $\Psi_{n} = \Psi_{*}$ for any $n \geq n'$. The inequality \eqref{dec:new-inq} yields for any $n \geq n'+1$
	\begin{equation*}
	x_{n+1} - x_{n} = 0, \ y_{n+1} - y_{n} = 0 \ \mbox{and} \ u_{n+1} - u_{n} = 0 .
	\end{equation*}
	The inequality \eqref{lim:z} gives us further $z_{n+1} - z_{n} = 0$ for any $n \geq n'+2$. This proves \eqref{conv:Cauchy}. 
	
	\item[Case 2.]
	Consider now the case when $\E_{n} > 0$ or, equivalently, $\Psi_{n} > \Psi_{*}$ for any $n \geq 1$. According to Lemma \ref{lem:uniformized}, there exist $\varepsilon > 0$, $\eta > 0$ and a desingularization function $\varphi$ such that for any element $\X$ in the intersection
	\begin{equation}
	\label{conv:intersection}
	\left\lbrace \Z \in \bR \colon \dist \left( \Z , \Omega \right) < \varepsilon \right\rbrace \cap
	\left\lbrace \Z \in \bR \colon \Psi_{*} < \Psi \left( \Z \right) < \Psi_{*} + \eta \right\rbrace
	\end{equation}
	it holds
	\begin{equation*}
	\varphi' \left( \Psi \left( \X \right) - \Psi_{*} \right) \cdot \dist \left( \0 , \partial \Psi \left( \X \right) \right) \geq 1.
	\end{equation*}
	Let be $n_{1} \geq 1$ such that for any $n \geq n_{1}$
	\begin{equation*}
	\Psi_{*} < \Psi_{n} < \Psi_{*} + \eta .
	\end{equation*}	
	Since $\lim\limits_{n \to +\infty} \dist \left( \X_{n} , \Omega \right) = 0$ (see Lemma \ref{lem:clus} \ref{lem:clus:iii}), there exists $n_{2} \geq 1$ such that for any $n \geq n_{2}$
	\begin{equation*}
	\dist \left( \X_{n} , \Omega \right) < \varepsilon .
	\end{equation*}
Consequently, $\X_n = \left( x_{n} , y_{n} , z_{n} , u_{n} , x_{n-1} , u_{n-1} \right)$ belongs to the intersection in \eqref{conv:intersection} for any $n \geq n_{0} := \max \left\lbrace n_{1} , n_{2} \right\rbrace$, which further implies
	\begin{equation}
	\label{conv:KL-property}
	\varphi ' \left( \Psi_{n} - \Psi_{*} \right) \cdot \dist \left( \0 , \partial \Psi \left( \X_{n} \right) \right) = \varphi ' \left( \E_{n} \right) \cdot \dist \left( \0 , \partial \Psi \left( \X_{n} \right) \right) \geq 1 .
	\end{equation}	
Define for two arbitrary nonnegative integers $i$ and $j$
	\begin{equation*}
	\Delta_{i,j} := \varphi \left( \Psi_{i} - \Psi_{*} \right) - \varphi \left( \Psi_{j} - \Psi_{*} \right) = \varphi \left( \E_{i} \right) - \varphi \left( \E_{j} \right) .
	\end{equation*}	
	The monotonicity of the sequence $\left\lbrace \Psi_{n} \right\rbrace _{n \geq 0}$ and of the function  $\varphi$ implies that $\Delta_{i,j} \geq 0$ for any $1 \leq i \leq j$.
	In addition, for any $N \geq n_{0} \geq 1$ it holds
	\begin{equation*}
	\mysum_{n = n_{0}}^{N} \Delta_{n,n+1} = \Delta_{n_{0},N+1} = \varphi \left( \E_{n_{0}} \right) - \varphi \left( \E_{N+1} \right) \leq  \varphi \left( \E_{n_{0}} \right),
	\end{equation*}
	from which we get $\mysum_{n \geq 1} \Delta_{n,n+1} < + \infty$.	
	
	By combining Lemma \ref{lem:dec} with the concavity of $\varphi$ we obtain for any $n \geq 1$
	\begin{align*}
	\begin{split}
	\Delta_{n,n+1} & = \varphi \left( \E_{n} \right) - \varphi \left( \E_{n+1} \right) \geq \varphi ' \left( \E_{n} \right) \left (\E_{n} - \E_{n+1} \right) = \varphi ' \left( \E_{n} \right) \left(\Psi_{n} - \Psi_{n+1} \right) \\
	& \geq \min \left\lbrace \Cdecx , \Cdecy , \Cdecu \right\rbrace \varphi ' \left( \E_{n} \right) \left( \left\lVert x_{n+1} - x_{n} \right\rVert ^{2} + \left\lVert y_{n+1} - y_{n} \right\rVert ^{2} + \left\lVert u_{n+1} - u_{n} \right\rVert ^{2} \right) .
	\end{split}
	\end{align*}
Thus, \eqref{conv:KL-property} implies for any $n \geq n_{0}$
	\begin{align*}
	& \ \left\lVert x_{n+1} - x_{n} \right\rVert ^{2} + \left\lVert y_{n+1} - y_{n} \right\rVert ^{2} + \left\lVert u_{n+1} - u_{n} \right\rVert ^{2} \\
	\leq & \ \dist \left( \0 , \partial \Psi \left( \X_{n} \right) \right) \cdot \varphi ' \left( \E_{n} \right) \left( \left\lVert x_{n+1} - x_{n} \right\rVert ^{2} + \left\lVert y_{n+1} - y_{n} \right\rVert ^{2} + \left\lVert u_{n+1} - u_{n} \right\rVert ^{2} \right) \\
	\leq & \ \Cmin \cdot \dist \left( \0 , \partial \Psi \left( \X_{n} \right) \right) \cdot \Delta_{n,n+1}.
	\end{align*}
	
	By the Cauchy-Schwarz inequality, the arithmetic mean-geometric mean inequality and Lemma \ref{lem:sgra}, we have that for any $n \geq n_{0}$ and every $\alpha >0$
	\begin{align}
	\label{conv:inq}
	& \ \left\lVert x_{n+1} - x_{n} \right\rVert + \left\lVert y_{n+1} - y_{n} \right\rVert + \left\lVert u_{n+1} - u_{n} \right\rVert \nonumber \\
	 \leq & \ \sqrt{3} \cdot \sqrt{\left\lVert x_{n+1} - x_{n} \right\rVert ^{2} + \left\lVert y_{n+1} - y_{n} \right\rVert ^{2} + \left\lVert u_{n+1} - u_{n} \right\rVert ^{2}} \nonumber \\
	 \leq & \ \sqrt{3 \Cmin} \cdot \sqrt{\dist \left( \0 , \partial \Psi \left( \X_{n} \right) \right) \cdot \Delta_{n,n+1}} \nonumber \\
	 \leq & \ \alpha \cdot \dist \left( \0 , \partial \Psi \left( \X_{n} \right) \right) + \dfrac{3 \Cmin}{4 \alpha} \Delta_{n,n+1} \nonumber \\
	 \leq & \ \alpha \Cmax \left( \left\lVert x_{n} - x_{n-1} \right\rVert + \left\lVert y_{n} - y_{n-1} \right\rVert + \left\lVert u_{n} - u_{n-1} \right\rVert \right) + \dfrac{3 \Cmin}{4 \alpha} \Delta_{n,n+1}.
	\end{align}
	If we denote for any $n \geq 0$
	\begin{equation}
	\label{conv:def}
	a_{n} := \left\lVert x_{n} - x_{n-1} \right\rVert + \left\lVert y_{n} - y_{n-1} \right\rVert + \left\lVert u_{n} - u_{n-1} \right\rVert \qquad \textrm{ and } \qquad b_{n} := \dfrac{3 \Cmin}{4 \alpha} \Delta_{n,n+1} ,
	\end{equation}
	then the above inequality is nothing else than \eqref{sum:hypo} with
	\begin{equation*}
	\chi_{0} := \alpha \Cmax \qquad \textrm{ and } \qquad \chi_{1} := 0.
	\end{equation*}
Since $\mysum_{n \geq 1} b_{n} < + \infty$, by choosing $\alpha < 1/ \Cmax$, we can apply Lemma \ref{lem:conv-ext} to conclude that
	\begin{equation*}
	\mysum_{n \geq 0} \Big( \left\lVert x_{n+1} - x_{n} \right\rVert + \left\lVert y_{n+1} - y_{n} \right\rVert + \left\lVert u_{n+1} - u_{n} \right\rVert \Big) < + \infty .
	\end{equation*}
The proof of \eqref{conv:Cauchy} is completed by taking into account once again \eqref{lim:z}.

From (i) it follows that the sequence $\left\lbrace \left( x_{n} , y_{n} , z_{n} , u_{n} \right) \right\rbrace _{n \geq 0}$ is Cauchy, thus it converges to an element $\left( x_{*}, y_{*}, z_{*}, u_{*} \right)$ which is, according to Lemmas \ref{lem:clus}, a KKT point of the optimization problem \eqref{intro:pb}.
\end{proof}
	
\subsection{Convergence rates}
\label{sec:rates}

In this section we derive convergence rates for the sequence $\left\lbrace \left( x_{n} , y_{n} , z_{n} , u_{n} \right) \right\rbrace _{n \geq 0}$ generated by Algorithm \ref{algo:PALM} as well as for $\left\lbrace \Psi_n \right\rbrace _{n \geq 0}$, if the regularized augmented Lagrangian $\Psi$  satisfies the \Loja property. The following definition is from \cite{Attouch-Bolte} (see also \cite{Lojasiewicz}).

\begin{defi}
	\label{defi:Lojasiewicz}
	Let $\Psi \colon \sR^{d} \to \sR \cup \left\lbrace + \infty \right\rbrace$ be proper and lower semicontinuous. Then $\Psi$ satisfies the \Loja property, if for any critical point $\widehat{v}$ of $\Psi$ there exists $C_{L}>0$, $\theta \in \left[ 0 , 1 \right)$ and $\varepsilon > 0$ such that
	\begin{equation*}
	\left\lvert \Psi \left( v \right) - \Psi \left( \widehat{v} \right) \right\rvert ^{\theta} \leq C_{L} \cdot \dist \left( 0 , \partial \Psi(v) \right) \ \forall v \in \Ball \left( \widehat{v} , \varepsilon \right),
	\end{equation*}
	where $\Ball \left( \widehat{v}, \varepsilon \right)$ denotes the open ball with center $\widehat{v}$ and radius $\varepsilon$.
\end{defi}	

If Assumption \ref{as} is fulfilled and $\left\lbrace \left( x_{n} , y_{n} , z_{n} , u_{n} \right) \right\rbrace _{n \geq 0}$ is the sequence generated by Algorithm \ref{algo:PALM}, then, according to Theorem \ref{lem:clus}, the set of cluster points $\Omega$ is nonempty, compact and connected and $\Psi$ takes on $\Omega$ the value $\Psi_{*}$; in addition, $\Omega \subseteq \crit \left( \Psi \right)$. 

According to \cite[Lemma 1]{Attouch-Bolte}, if $\Psi$ has the \Loja property, then there exist $C_{L} > 0$, $\theta \in \left[ 0 , 1 \right)$ and  $\varepsilon > 0$ such that for any
\begin{equation*}
\X \in \left\lbrace \Z \in \bR \colon \dist \left( \Z , \Omega \right) < \varepsilon \right\rbrace,
\end{equation*}
it holds
\begin{equation*}
\left\lvert \Psi \left( \X \right) - \Psi_{*} \right\rvert ^{\theta} \leq C_{L} \cdot \dist \left( 0 , \partial \Psi \left( \X \right) \right).
\end{equation*}

Obviously, $\Psi$ is a \KL \ function with desingularization function 
\begin{equation*}
\varphi : [0,+\infty) \to [0,+\infty), \ \varphi \left( s \right) := \dfrac{1}{1 - \theta} C_{L} s^{1 - \theta},
\end{equation*}
which, according to Theorem \ref{thm:conv}, means that $\Omega$ contains a single element $\X_{*}$, which is the limit of $\left\lbrace \X_{n} \right\rbrace _{n \geq 1}$ as $n \to +\infty$. In other words, if $\Psi$ has the \Loja property, then there exist $C_{L} > 0$, $\theta \in \left[ 0 , 1 \right)$ and  $\varepsilon > 0$ such that for any $\X \in \Ball \left( \X_*  , \varepsilon \right)$
\begin{equation}
\label{Loja:uniform}
\left\lvert \Psi \left( \X \right) - \Psi_{*} \right\rvert ^{\theta} \leq C_{L} \cdot \dist \left( 0 , \partial \Psi \left( \X \right) \right).
\end{equation}
In this case, $\Psi$ is said to satisfy the \Loja property with \Loja constant $C_{L} > 0$ and \Loja exponent $\theta \in \left[ 0 , 1 \right)$.

The following lemma will provide convergence rates for a particular class of monotonically decreasing real sequences converging to $0$. Its proof can be found in \cite[Lemma 15]{Bot-Nguyen}.
\begin{lem}
	\label{lem:rates}
	Let $\left\lbrace e_{n} \right\rbrace _{n \geq 0}$ be a monotonically decreasing sequence of nonnegative numbers converging $0$. Assume further that there exists natural numbers $n_{0} \geq 1$ such that for any $n \geq n_{0}$
	\begin{equation*}
	e_{n-1} - e_{n} \geq C_{e} e_{n}^{2 \theta},
	\end{equation*}
	where $C_{e} > 0$ is some constant and $\theta \in \left[ 0 , 1 \right)$. The following statements are true:
	\begin{enumerate}
		\item
		\label{lem:rates:i}
		if $\theta = 0$, then $\left\lbrace e_{n} \right\rbrace _{n \geq 0}$ converges in finite time;
		
		\item
		\label{lem:rates:ii}
		if $\theta \in \left( 0 , 1/2 \right]$, then there exist $C_{e,0} > 0$ and $Q \in \left[ 0 , 1 \right)$ such that for any $n \geq n_{0}$
		\begin{equation*}
		0 \leq e_{n} \leq C_{e,0} Q^{n};
		\end{equation*}
		
		\item
		\label{lem:rates:iii}
		if $\theta \in \left( 1/2 , 1 \right)$, then there exists $C_{e,1} > 0$ such that for any $n \geq n_{0}+1$
		\begin{equation*}
		0 \leq e_{n} \leq C_{e,1} n^{- \frac{1}{2 \theta - 1}}  .
		\end{equation*}
	\end{enumerate}
\end{lem}

We prove a recurrence inequality for the  sequence $\left\lbrace \E_{n} \right\rbrace _{n \geq 0}$.
\begin{lem}
	\label{lem:rec}
Let Assumption \ref{as} be satisfied and $\left\lbrace \left( x_{n} , y_{n} , z_{n} , u_{n} \right) \right\rbrace _{n \geq 0}$ be a sequence generated by Algorithm \ref{algo:PALM}.  If $\Psi$ satisfies the \Loja property with \Loja constant $C_{L} > 0$ and \Loja exponent $\theta \in \left[ 0 , 1 \right)$, then there exists $n_{0} \geq 1$ such that the following estimate holds for any $n \geq n_{0}$ 
	\begin{equation}
	\label{rec:inq}
	\E_{n-1} - \E_{n} \geq \Crec \E_{n}^{2 \theta} , \qquad \textrm{ where } \quad \Crec := \dfrac{\Cmin}{3 \left( C_{L} \cdot \Cmax \right) ^{2}}.
	\end{equation}
\end{lem}
\begin{proof}
	For every $n \geq 2$ we obtain from Lemma \ref{lem:dec}
	\begin{align*}
	\E_{n-1} - \E_{n} & = \Psi_{n-1} - \Psi_{n} \\
	& \geq \Cmin \left( \left\lVert x_{n} - x_{n-1} \right\rVert ^{2} + \left\lVert y_{n} - y_{n-1} \right\rVert ^{2} + \left\lVert u_{n} - u_{n-1} \right\rVert ^{2} \right) \\
	& \geq \dfrac{1}{3} \Cmin \left( \left\lVert x_{n} - x_{n-1} \right\rVert + \left\lVert y_{n} - y_{n-1} \right\rVert + \left\lVert u_{n} - u_{n-1} \right\rVert \right) ^{2} \\
	& \geq \Crec C_{L}^{2} \opnorm{D_{n}} ^{2},
	\end{align*}
where $D_n \in \partial \Psi(\X_n)$.
	Let $\varepsilon > 0$ be such that \eqref{Loja:uniform} is fulfilled and choose $n_{0} \geq 1$ with the property that for any $n \geq n_{0}$, $X_{n}$ belongs to $\Ball ( \X_{*} , \varepsilon )$. Relation \eqref{Loja:uniform} implies \eqref{rec:inq} for any $n \geq n_{0}$.
\end{proof}
		
The following result follows by combining Lemma \ref{lem:rates} with Lemma \ref{lem:rec}.
\begin{thm}
	\label{thm:obj}
Let Assumption \ref{as} be satisfied and $\left\lbrace \left( x_{n} , y_{n} , z_{n} , u_{n} \right) \right\rbrace _{n \geq 0}$ be a sequence generated by Algorithm \ref{algo:PALM}.  If $\Psi$ satisfies the \Loja property with \Loja constant $C_{L} > 0$ and \Loja exponent $\theta \in \left[ 0 , 1 \right)$, then the following statements are true:
	\begin{enumerate}
		\item 
		if $\theta = 0$, then $\left\lbrace \Psi_{n} \right\rbrace _{n \geq 1}$ converges in finite time;
		
		\item 
		if $\theta \in \left( 0 , 1/2 \right]$, then there exist $n_{0} \geq 1$, $\widehat{C}_{0} > 0$ and $Q \in \left[ 0 , 1 \right)$ such that for any $n \geq n_{0}$
		\begin{equation*}
		0 \leq \Psi_{n} - \Psi_{*} \leq \widehat{C}_{0} Q^{n};
		\end{equation*}
		
		\item
		if $\theta \in \left( 1/2 , 1 \right)$, then there exist $n_{0} \geq 1$ and $\widehat{C}_{1} > 0$ such that for any $n \geq n_{0} + 1$
		\begin{equation*}
		0 \leq \Psi_{n} - \Psi_{*} \leq \widehat{C}_{1} n ^ {- \frac{1}{2 \theta - 1}}  .
		\end{equation*}
	\end{enumerate}
\end{thm}

The next lemma will play an important role when transferring the convergence rates for 
$\left\lbrace \Psi_{n} \right\rbrace _{n \geq 0}$ to the sequence of iterates 
$\left\lbrace \left( x_{n} , y_{n} , z_{n} , u_{n} \right) \right\rbrace _{n \geq 0}$.
\begin{lem}
	\label{lem:ite}
	Let Assumption \ref{as} be satisfied and $\left\lbrace \left( x_{n} , y_{n} , z_{n} , u_{n} \right) \right\rbrace _{n \geq 0}$ be a sequence generated by Algorithm \ref{algo:PALM}. Let $\left( x_{*} , y_{*} , z_{*} , u_{*} \right)$ be the KKT point of the optimization problem \eqref{intro:pb} to which $\left\lbrace \left( x_{n} , y_{n} , z_{n} , u_{n} \right) \right\rbrace _{n \geq 0}$ converges as $n \to +\infty$. Then there exists $n_{0} \geq 1$ such that the following estimates hold for any $n \geq n_{0}$
	\begin{align}
	\label{ite:bound}
	\left\lVert x_{n} - x_{*} \right\rVert \leq \Citex \max \left\lbrace \sqrt{\E_{n}} , \varphi \left( \E_{n} \right) \right\rbrace , & \quad \left\lVert y_{n} - y_{*} \right\rVert \leq \Citex \max \left\lbrace \sqrt{\E_{n}} , \varphi \left( \E_{n} \right) \right\rbrace, \nonumber \\
	\left\lVert z_{n} - z_{*} \right\rVert \leq \Citez \max \left\lbrace \sqrt{\E_{n}} , \varphi \left( \E_{n} \right) \right\rbrace , & \quad \left\lVert u_{n} - u_{*} \right\rVert \leq \Citex \max \left\lbrace \sqrt{\E_{n}} , \varphi \left( \E_{n} \right) \right\rbrace,
	\end{align}
	where
	\begin{equation*}
	\Citex 	:=2\sqrt{3\Cmin}+3\Cmin\Cmax \qquad \textrm{ and } \qquad \Citez := \left( \left\lVert A \right\rVert + \dfrac{2}{\sigma \beta} \right) \Citex .
	\end{equation*}
\end{lem}
\begin{proof}
We assume that $\E_{n} > 0$ for any $n \geq 0$. Otherwise, the sequence $\left\lbrace \left( x_{n} , y_{n} , z_{n} , u_{n} \right) \right\rbrace _{n \geq 0}$ becomes identical to $\left( x_{*} , y_{*} , z_{*} , u_{*} \right)$ beginning with a given index and the conclusion follows automatically (see the proof of Theorem \ref{thm:conv}).
	
Let $\varepsilon > 0$ be such that \eqref{Loja:uniform} is fulfilled and $n_{0} \geq 2$ be such that $X_{n}$ belongs to $\Ball ( \X_{*} , \varepsilon )$ for any $n \geq n_{0}$. 
	
We fix $n \geq n_{0}$ now. One can easily notice that
	\begin{equation*}
	\left\lVert x_{n} - x_{*} \right\rVert \leq \left\lVert x_{n+1} - x_{n} \right\rVert + \left\lVert x_{n+1} - x_{*} \right\rVert \leq \cdots \leq \mysum_{k \geq n} \left\lVert x_{k+1} - x_{k} \right\rVert.
	\end{equation*}
Similarly, we derive
	\begin{equation*}
	\left\lVert y_{n} - y_{*} \right\rVert \leq \mysum_{k \geq n} \left\lVert y_{k+1} - y_{k} \right\rVert , \quad \left\lVert z_{n} - z_{*} \right\rVert \leq \mysum_{k \geq n} \left\lVert z_{k+1} - z_{k} \right\rVert , \quad \left\lVert u_{n} - u_{*} \right\rVert \leq \mysum_{k \geq n} \left\lVert u_{k+1} - u_{k} \right\rVert .
	\end{equation*}
On the other hand, in view of \eqref{conv:def} and by taking $\alpha := \dfrac{1}{2 \Cmax}$ the inequality \eqref{conv:inq} can be written as 
	\begin{equation*}
	a_{n+1} \leq \dfrac{1}{2} a_{n} + b_{n} \ \forall n \geq n_0.
	\end{equation*}
Let us fix now an integer $N \geq n$. Summing up the above inequality for $k = n, ..., N$, we have
	\begin{align*}
	\mysum_{k=n}^{N} a_{k+1} & \leq \dfrac{1}{2} \mysum_{k = n}^{N} a_{k} + \mysum_{k=n}^{N} b_{k} = \dfrac{1}{2} \mysum_{k=n}^{N} a_{k+1} + a_{n} - a_{N+1} + \mysum_{k=n}^{N} b_{k} \\
	& \leq \dfrac{1}{2} \mysum_{k=n}^{N} a_{k+1} + a_{n} + \dfrac{3\Cmin  \Cmax}{2} \varphi \left( \E_{n} \right).
	\end{align*}
	By passing $N \to + \infty$, we obtain
	\begin{align*}
	\begin{split}
	\mysum_{k \geq n} a_{k+1} & = \mysum_{k \geq n} \left( \left\lVert x_{k+1} - x_{k} \right\rVert + \left\lVert y_{k+1} - y_{k} \right\rVert + \left\lVert u_{k+1} - u_{k} \right\rVert \right) \\
	& \leq 2 \left( \left\lVert x_{n+1} - x_{n} \right\rVert + \left\lVert y_{n+1} - y_{n} \right\rVert + \left\lVert u_{n+1} - u_{n} \right\rVert \right) + 3\Cmin  \Cmax\varphi \left( \E_{n} \right) \\
	& \leq 2 \sqrt{3} \cdot \sqrt{\left\lVert x_{n+1} - x_{n} \right\rVert ^{2} + \left\lVert y_{n+1} - y_{n} \right\rVert ^{2} + \left\lVert u_{n+1} - u_{n} \right\rVert ^{2}} + 3\Cmin  \Cmax\varphi \left( \E_{n} \right) \\
	& \leq 2 \sqrt{3 \Cmin} \cdot \sqrt{\E_{n} - \E_{n+1}} + 3\Cmin  \Cmax\varphi \left( \E_{n} \right) ,
	\end{split}
	\end{align*}
	which gives the desired statement.
\end{proof}

We can now formulate convergence rates for the sequence of generated iterates.
\begin{thm}
	\label{thm:ite}
	Let Assumption \ref{as} be satisfied and $\left\lbrace \left( x_{n} , y_{n} , z_{n} , u_{n} \right) \right\rbrace _{n \geq 0}$ be a sequence generated by Algorithm \ref{algo:PALM}. Suppose further that $\Psi$ satisfies the \Loja property with \Loja constant $C_{L} > 0$ and \Loja exponent $\theta \in \left[ 0 , 1 \right)$. Let $\left( x_{*} , y_{*} , z_{*} , u_{*} \right)$ be the KKT point of the optimization problem \eqref{intro:pb} to which $\left\lbrace \left( x_{n} , y_{n} , z_{n} , u_{n} \right) \right\rbrace _{n \geq 0}$ converges as $n \to +\infty$. Then the following statements are true:
	\begin{enumerate}
		\item
		if $\theta = 0$, then the algorithm converges in finite time;
		
		\item
		if $\theta \in \left( 0 , 1/2 \right]$, then there exist $n_{0} \geq 1$, $\widehat{C}_{0,1} , \widehat{C}_{0,2} , \widehat{C}_{0,3} , \widehat{C}_{0,4} > 0$ and $\widehat Q \in \left[0 , 1 \right)$ such that for any $n \geq n_{0}$
		\begin{equation*}
		\left\lVert x_{n} - x_{*} \right\rVert \leq \widehat{C}_{0,1} \widehat Q^{k} , \quad \left\lVert y_{n} - y_{*} \right\rVert \leq \widehat{C}_{0,2} \widehat Q^{k} , \quad \left\lVert z_{n} - z_{*} \right\rVert \leq \widehat{C}_{0,3} \widehat Q^{k} , \quad \left\lVert u_{n} - u_{*} \right\rVert \leq \widehat{C}_{0,4} \widehat Q^{k};
		\end{equation*}
		
		\item
		if $\theta \in \left( 1/2 , 1 \right)$, then there exist $n_{0} \geq 1$ and $\widehat{C}_{1,1} , \widehat{C}_{1,2} , \widehat{C}_{1,3} , \widehat{C}_{1,4} > 0$ such that for any $n \geq n_{0} + 1$
		\begin{align*}
		\left\lVert x_{n} - x_{*} \right\rVert \leq \widehat{C}_{1,1} n^{- \frac{1-\theta}{2 \theta - 1}} , & \quad \left\lVert y_{n} - y_{*} \right\rVert \leq \widehat{C}_{1,2} n^{- \frac{1-\theta}{2 \theta - 1}} , \\
		\left\lVert z_{n} - z_{*} \right\rVert \leq \widehat{C}_{1,3} n^{- \frac{1-\theta}{2 \theta - 1}} , & \quad \left\lVert u_{n} - u_{*} \right\rVert \leq \widehat{C}_{1,4} n^{- \frac{1-\theta}{2 \theta - 1}}.
		\end{align*}
	\end{enumerate}
\end{thm}
\begin{proof}
Let 
	\begin{equation*}
	\varphi : [0,+\infty) \to [0,+\infty), \quad s \mapsto \dfrac{1}{1 - \theta} C_{L} s^{1 - \theta},
	\end{equation*}
	be the desingularization function.

\item[(i)] If $\theta =0$, then $\left\lbrace \Psi_{n} \right\rbrace _{n \geq 1}$ converges in finite time. 
		As seen in the proof of Theorem \ref{thm:conv}, the sequence $\left\lbrace \left( x_{n} , y_{n} , z_{n} , u_{n} \right) \right\rbrace _{n \geq 0}$ becomes identical to $\left( x_{*} , y_{*} , z_{*} , u_{*} \right)$ starting from a given index. In other words, the sequence $\left\lbrace \left( x_{n} , y_{n} , z_{n} , u_{n} \right) \right\rbrace _{n \geq 0}$ converges also in finite time and the conclusion follows.
		
Let be $\theta \neq \frac{1}{2}$ and $n_{0}' \geq 1$ such that for any $n \geq n_{0}'$ the inequalities \eqref{ite:bound} in Lemma \ref{lem:ite} and 
	\begin{equation*}
	\E_{n} \leq \left ( \dfrac{1}{1 - \theta} C_{L} \right)^{\frac{2}{2 \theta - 1}}
	\end{equation*}
hold.

\item[(ii)] If $\theta \in \left( 0 , 1/2 \right)$, then $2 \theta - 1 < 0$ and thus for any $n \geq n_{0}'$
		\begin{equation*}
		\dfrac{1}{1 - \theta} C_{L} \E_{n}^{1 - \theta} \leq \sqrt{\E_{n}} ,
		\end{equation*}
		which implies that		
		\begin{equation*}
		\max \left\lbrace \sqrt{\E_{n}} , \varphi \left( \E_{n} \right) \right\rbrace = \sqrt{\E_{n}} .
		\end{equation*}
If $\theta = 1/2$, then
		\begin{equation*}
		\varphi \left( \E_{n} \right) = 2 C_{L} \sqrt{\E_{n}},
		\end{equation*}
thus
		\begin{equation*}
		\max \left\lbrace \sqrt{\E_{n}} , \varphi \left( \E_{n} \right) \right\rbrace = \max \left\lbrace 1 , 2 C_{L} \right\rbrace \cdot \sqrt{\E_{n}} \ \forall n \geq 1.
		\end{equation*}
In both cases we have
		\begin{equation*}
		\max \left\lbrace \sqrt{\E_{n}} , \varphi \left( \E_{n} \right) \right\rbrace \leq \max \left\lbrace 1 , 2 C_{L} \right\rbrace \cdot \sqrt{\E_{n}} \ \forall n \geq n_0'.
		\end{equation*}
		By Theorem \ref{thm:obj}, there exist $n_{0}'' \geq 1$, $\widehat{C}_{0} > 0$ and $Q \in [0,1)$ such that for $\widehat{Q} := \sqrt{Q}$ and every $n \geq n_{0}''$ it holds
		\begin{equation*}
		\sqrt{\E_{n}} \leq \sqrt{\widehat{C}_{0}} Q^{n/2} = \sqrt{\widehat{C}_{0}} \widehat{Q}^{n}.
		\end{equation*}
		The conclusion follows from Lemma \ref{lem:ite} for $n_{0} := \max \left\lbrace n_{0}', n_{0}'' \right\rbrace$.

\item[(iii)] If $\theta \in \left( 1/2 , 1 \right)$, then $2 \theta - 1 > 0$ and thus for any $n \geq n_{0}'$
		\begin{equation*}
		\sqrt{\E_{n}} \leq \dfrac{1}{1 - \theta} C_{L} \E_{n}^{1 - \theta},
		\end{equation*}
		which implies that		
		\begin{equation*}
		\max \left\lbrace \sqrt{\E_{n}} , \varphi \left( \E_{n} \right) \right\rbrace = \varphi \left( \E_{n} \right) = \dfrac{1}{1 - \theta} C_{L} \E_{n}^{1 - \theta}.
		\end{equation*} 
		By Theorem \ref{thm:obj}, there exist $n_{0}'' \geq 1$ and $\widehat{C}_{1} > 0$ such that for any $n \geq n_{0}''$
		\begin{equation*}
		\dfrac{1}{1 - \theta} C_{L} \E_{n}^{1 - \theta} \leq\dfrac{1}{1 - \theta} C_{L} \widehat{C}_{1}^{1-\theta} \left( n - 2 \right) ^{- \frac{1 - \theta}{2 \theta - 1}} .
		\end{equation*}
		The conclusion follows again for $n_{0} := \max \left\lbrace n_{0}', n_{0}'' \right\rbrace$ from Lemma \ref{lem:ite}.
		\qedhere
\end{proof}

{\bf Ackonwledgements.} The authors are thankful to two anonymous referees for their comments and recommendations which improved the quality of the paper.

%
%
%
%

\end{document}